\documentclass[reqno, 12pt]{amsart}

\usepackage[text={160mm,240mm},centering]{geometry}

\input diagxy

\input xy

\usepackage[active]{srcltx} 

\newtheorem{thm}{Theorem}[section]

\newtheorem{lem}[thm]{Lemma}
\newtheorem{prop}[thm]{Proposition}

\theoremstyle{definition}
\newtheorem{defn}[thm]{Definition}
\theoremstyle{remark}
\newtheorem{rem}[thm]{Remark}

\newtheorem{ex}[thm]{Example}
\numberwithin{equation}{section}
\usepackage[mathscr]{eucal}
\usepackage[all]{xy}
\usepackage{mathrsfs}
\usepackage{xypic}
\usepackage{amsfonts}
\usepackage{amsmath}
\usepackage{amsthm}
\usepackage{amssymb}
\usepackage{latexsym}
\usepackage{tabularx}
\usepackage{graphicx}
\usepackage{pict2e}
\usepackage{hyperref}
\begin{document}

\title[$G$-vector bundle, Euler cycle]
 { Stratified obstruction systems for equivariant moduli problems and invariant Euler cycles}
 \author{Xiangdong Yang}
\address{Department of Mathematics\\
Sichuan University\\\newline
 Chengdu 610064 P. R. China}
\email{xiangdongyang2009@gmail.com}

\subjclass[2010]{57R22; 57R91}

\keywords{Equivariant vector bundle, equivariant moduli problem, Euler cycle }
\date{December 30, 2013.}




\begin{abstract}
The purpose of this paper is to study finite dimensional equivariant moduli problems from the viewpoint of stratification theory.
We show that there exists a stratified obstruction system for a finite dimensional equivariant moduli problem.
In addition, we define a coindex for a $G$-vector bundle which is determined by the $G$-action on the vector bundle and prove that if the coindex of an oriented equivariant moduli problem is bigger than $1$, then we obtain an invariant Euler cycle via equivariant perturbation.
In particular, we get a localization formula for the stratified transversal intersection of $S^{1}$-moduli problems.
\end{abstract}

\maketitle

\tableofcontents
\section{Introduction}

Assume that $\pi:E\longrightarrow B$ is an oriented smooth vector bundle of rank $k$ over an $n$-dimensional closed manifold $B$
and $S:B\longrightarrow E$ is a smooth section.
The zero locus $S^{-1}(0)\subset B$ contains a lot of topological information of the bundle $E$ when it is transversal to the zero section,
and those information is reflected as a cycle, in $B$, which is called the Euler cycle of $E$.
Furthermore, the Euler cycle represents a homology class which is a topological invariant of $E$.
For an oriented vector bundle, there are two dual viewpoints to construct such topological invariant.
 \begin{enumerate}
   \item  For any smooth section $S:B\longrightarrow E$ (not necessarily transversal) via a slightly smooth perturbation we get a smooth section $S+P$ which is transverse to the zero section and hence the zero locus $(S+P)^{-1}(0)\subset B$ is an oriented closed submanifold with dimension $n-k$.
       Moreover, this submanifold yields a homology class $[(S+P)^{-1}(0)]\in H_{n-k}(B;\mathbb{Z}),$
       which is independent of the choice of such perturbations.
   \item There exists a unique cohomology class on $E$, denoted by $\Theta\in H^{k}_{vc}(E)$ (vertical compact cohomology), that restricts to the generator of $H^{k}_{c}(F)$ (compact cohomology) on each fiber $F$.
       This class is called the Thom class of $E$,  and the cohomological Euler class of $E$ is defined as the pullback of the Thom class by the zero section.

 \end{enumerate}

 Now let us consider the equivariant case.
 Given a finite group $\Gamma$, if $\pi:E\longrightarrow B$ is a smooth $\Gamma$-equivariant vector bundle and $S$ is a smooth $\Gamma$-equivariant section, then the technique of perturbation for the section $S$ is ineffective since there is no $\Gamma$-equivariant smooth perturbed section of $S$ which is transversal to the zero section in general.
 Since $\Gamma$ is finite, Fukaya-Ono \cite{FO99} constructed a multi-valued perturbation such that each branch is transverse to the zero section.
 Moreover, the zero locus of such multisection gives rise to a cycle over the rational numbers $\mathbb{Q}$ and through this technique they defined the Euler class of an oriented orbibundle.
 In particular, for an oriented orbibundle $E\longrightarrow X$ with a smooth section $S$ having compact locus $S^{-1}(0)$, Lu-Tian \cite{LT07} gave a very thorough presentation of constructing multi-valued perturbation of $S$ by gluing together sections which resolute the fiber products over the local uniformizing charts of $S^{-1}(0)$.
 Moreover, they showed that the zero locus of the perturbed section yields a rational Euler cycle.

Generally, let $G$ be a connected compact Lie group and $B$ be a smooth manifold on which $G$ acts smoothly and effectively.
Suppose that $\pi:E\longrightarrow B$ is an oriented smooth $G$-vector bundle.
Using the equivariant cohomology theory, Mathai-Quillen constructed an equivariant Thom class $\Theta_{eq}$ of $E$ which is a compactly supported closed equivariant form such that its integral along the fibres is the constant function 1 on $B$ (cf. \cite{MQ86}).
Thus the equivariant Euler class of $E$ can be defined as the pullback of the equivariant Thom class by the zero section.
For a nontrivial action of a Lie group $G$ on the vector bundle $E$, it is generally unrealistic to require a smooth section to be equivariant and transversal. In fact, for a $G$-equivariant smooth section $S:B\longrightarrow E$, there exist some global obstructions (cf. \cite{SS92,TP74}) to deform $S$ to a $G$-equivariant smooth transversal section.
Therefore in the equivariant case the transversality is too rigid, and a natural problem is:
\emph{Can we define a new version of transversality in $G$-equivariant case such that}
\begin{itemize}
  \item \emph{every $G$-equivariant smooth section has an equivariant perturbation which satisfies this new transversality; and}
  \item \emph{how to construct the invariant Euler cycle of an oriented $G$-vector bundle?}
\end{itemize}

In fact, this is a \emph{finite dimensional equivariant moduli problem} in the sense of Cieliebak-Riera-Salamon (cf. \cite{CRS03}).
Independently, Bierstone and Field also discussed the transversality problem in the equivariant case.
In \cite{EB77}, Bierstone introduced the notion of general position for smooth $G$-equivariant maps between smooth $G$-manifolds.
Meanwhile, Field in \cite{MF177} proposed the concept of $G$-transversality and showed that the two definitions are equivalent (cf. \cite{MF377}).
An infinite dimensional version of equivariant general position was defined by Hambleton-Lee (cf. \cite{HL92}).
Furthermore they studied the equivariant perturbation of Yang-Mills moduli space with a compact Lie group action.

For an oriented finite dimensional $G$-moduli problem, which is regular, i.e. the isotropy subgroup of $G$-action on $B$ is finite, Cieliebak-Riera-Salamon constructed a rational cycle through multi-valued perturbation in the paper \cite{CRS03}.
The method of Cieliebak-Riera-Salamon is similar to the one that Fukaya-Ono used in constructing the Euler class of an oriented orbibundle.
The main issue is that for an equivariant vector bundle with a finite group action we can not guarantee that the transversal perturbation of an equivariant section is also equivariant in general, so the perturbed section can not descend to a single-valued section of the associated orbibundle.
Therefore the multi-valued perturbation is necessary to obtain the transversality.

The method we use here is to perturb the equivariant smooth section in the sense of equivariant general position and to represent the fundamental class with a Whitney object rather than a smooth submanifold.
Using the geometric chain (cycle) technique introduced by Goresky, we obtain the following main theorem.
\begin{thm}\label{Main Thm1}
Let $(B,E,S)$ be a finite dimensional oriented $G$-moduli problem with $\emph{dim}\,B=n$ and $\emph{rank}\,E=k$.
If $\emph{coind}\,(B,E)>1$, then there exists a smooth equivariant perturbation
$P:B\longrightarrow E$ supported in an invariant open neighborhood of $S^{-1}(0)$,
such that $S+P$ is in general position with respect to the zero section.
Furthermore, the zero locus $(S+P)^{-1}(0)\subset B$ yields a $G$-invariant $(n-k)$-geometric cycle and it represents a homology class
$$
[(S+P)^{-1}(0)]\in H_{n-k}(B;\mathbb{Z}),
$$
which is independent of the choice of such perturbations.
\end{thm}
This paper is organized as follows.
We devote Section 2 to the preliminaries of the definition of Whitney stratified chains and the definition of general position for equivariant smooth maps.
In Section 3 we construct the obstruction system of a $G$-moduli problem.
Furthermore, we define a coindex $\textmd{coind}(B,E)$ for a $G$-vector bundle $\pi:E\longrightarrow B$.
In Section 4 we give the proof of Theorem \ref{Main Thm1}.
Finally, in Section 5 we study the transversal intersection of $S^{1}$-moduli problems
and show that all geometric information of transversal intersection is contained in the fixed submanifold of the $S^{1}$-action (Theorem \ref{Main Thm2}).

\subsection*{Acknowledgements}
I am greatly indebted to my supervisors Prof. Guosong Zhao and Prof. Xiaojun Chen for their constant encouragements,
supports and many useful discussions.
Particular thanks go to Prof. Bohui Chen who made numerous helpful suggestions.
I also thank Prof. Quan Zheng for guidance over the past years.
Finally, I am grateful to the anonymous referee for useful comments and suggestions.

\section{Preliminaries}

\subsection{Whitney stratification and geometric cycles}
The stratified space is motivated by the study of singular manifolds, which naturally arises in the study of algebraic, analytic varieties, and singularities of smooth mappings.
Intuitively, a stratified space is an object constructed by gluing some smooth manifolds with different dimensions nicely.

\begin{defn}
Let $X$ be a Hausdorff and paracompact topological space, and $\mathcal{I}$ be a poset with order relations denoted by $\leq$.
If $X$ is a locally finite collection of disjoint locally closed manifolds $S_{i}\subset X \,(i\in \mathcal{I})$ and satisfies the following conditions:
\begin{enumerate}
  \item $X=\cup_{i\in \mathcal{I}}S_{i}$;
  \item $S_{i}\cap \bar{S_{j}}\neq\emptyset\Longleftrightarrow S_{i}\subset \bar{S_{j}}\Longleftrightarrow i\leq j$;
\end{enumerate}
then the family $\mathcal{S}=\{S_{i}\subset X\mid i\in \mathcal{I}\}$ is called a \emph{stratification} of $X$,
and $(X,\mathcal{S})$ is called a \emph{stratified space}.
A piece $S_{i}\in \mathcal{S}$ is called a \emph{stratum} of $X$.

A \emph{stratified subspace} of $(X,\mathcal{S})$ is a subset $Y\subset X$
such that $$\mathcal{S}_{Y}=\{S\cap Y\mid S\in \mathcal{S}\}$$ is a stratification of $Y$ with the induced topology.
\end{defn}

If $S_{i}\subset \bar{S}_{j}$, write $S_{i}\leq S_{j}$.
If $S_{i}\leq S_{j}$ and $S_{i}\neq S_{j}$, write $S_{i}<S_{j}$.

\begin{defn}
Let $\mathcal{S}$ be a stratification of the space $X$.
The \emph{length of a stratum} $S\in \mathcal{S}$ is defined to be the integer
\[l_{X}(S):=\sup\{n\mid  S=S_{0}<S_{1}<\cdot\cdot\cdot<S_{n}\}\]
where $S_{1},...,S_{n}$ are strata of $X$.
The \emph{length of stratified space} $(X,\mathcal{S})$ is defined to be
\[l(X):=\sup_{i\in \mathcal{I}}l_{X}(S_{i}).\]
\end{defn}

A stratum $S\in\mathcal{S}$ is called \emph{maximal} (resp. \emph{minimal}) if it is open (resp. closed).
The \emph{dimension of a stratified space} $(X,\mathcal{S})$ is defined to be the dimension of the maximal stratum.
The stratum $S\in \mathcal{S}$ is called \emph{regular stratum} if it is open in $X$, otherwise it is called \emph{singular stratum}.
The union of all singular strata, denoted by $\Sigma$, is called the \emph{singular part} of $X$.
And the minimal part, denoted by $\Sigma_{min}$, is the union of minimal strata.
Let $X_{i}=\bigcup_{j\leq i}S_{j}$, and $X_{i}$ is called a \emph{skeleton} of $X$.
There exists a finite filtration of skeletons
\[X=X_{m}\supseteq X_{m-1}\supseteq\cdot\cdot\cdot\supseteq X_{0}\supseteq X_{-1}=\emptyset,\]
where $m$ is called the \emph{depth} of $X$.

\begin{ex}
 A smooth manifold $M$ is a stratified space with empty singular part $\Sigma=\emptyset$.
 \end{ex}

\begin{ex}
Let $V$ be a subset of $\mathbb{R}^{n}$.
Then $V$ is called an \emph{algebraic set} of $\mathbb{R}^{n}$, if it is
the common loci of finitely many real polynomials.
Note that the singular set $\Sigma V$ of all points where $V$ fails to be a smooth manifold is also an algebraic set, hence there is a finite filtration of $$V=V_{m}\supseteq V_{m-1}\supseteq\cdot\cdot\cdot\supseteq V_{0}\supseteq V_{-1}=\emptyset$$ with $V_{i-1}=\Sigma V_{i}$.
Clearly $V$ is a stratified space with stratum $S_{i}=V_{i}- V_{i-1}$.
\end{ex}

Inspired by the ideas of Thom on stratifications, Whitney introduced  \emph{Condition A} and \emph{Condition B} (cf. \cite{HW165}).
Actually, Condition B implies Condition A, this was proved by Mather in his lecture notes \cite{JM12}.
Therefore, in general we only use Condition B to define a Whitney stratification.
Let us recall the definition of Whitney's Condition B.
Given any $x,y\in \mathbb{R}^{n}$ such that $x\neq y$, the secant $\overset{\frown}{xy}$ is defined to be the line in $\mathbb{R}^{n}$ which is parallel to the line $\overline{xy}$ (line joining $x$ and $y$) and passed through the origin.
\begin{defn}
(Condition B for submanifolds of $\mathbb{R}^{n}$) Let $X$ and $Y$ be the smooth submanifolds of $\mathbb{R}^{n}$.
Assume that $\textmd{dim}X=r$.
We say that the pair $(X,Y)$ satisfies \emph{Condition B} at a given point $y\in Y$, if the following holds: Let $\{x_{i}\}$ and $\{y_{i}\}$ be two sequences of points in $X$ and $Y$ respectively, satisfying  $\{x_{i}\}$ and  $\{y_{i}\}$ converging to $y$.
Suppose that the tangent space $T_{x_{i}}X$ converges to some $r$-plane $\tau\subset \mathbb{R}^{n}$, and the secants $\overset{\frown}{x_{i}y_{i}}$ $(x_{i}\neq y_{i})$ converge to some line $l\subset \mathbb{R}^{n}$, then $l\subset\tau$.
\end{defn}

This definition can be extended to submanifolds of arbitrary smooth manifolds.

\begin{defn}
Let $M$ be a smooth $m$-manifold, $X$ and $Y$ be smooth submanifolds.
Given $y\in Y$, we say that the pair $(X,Y)$ satisfies \emph{Condition B} at $y$, if for some coordinate chart $(\varphi,U)$ about $y$, the pair
$$(\varphi(U\cap X),\varphi(U\cap Y))$$
satisfies \emph{Condition B} at $\varphi(y)$ in $\mathbb{R}^m$.
This definition is well-defined, as it is independent of the choice of the coordinate chart (cf. \cite{JM12}).
\end{defn}

\begin{ex}(\cite[Theorem 4.3.7]{JP01})
Let $G$ be a compact Lie group.
If $M$ is a smooth manifold on which $G$ acts smoothly, then the stratification by orbit types of $M$, denoted by
 $$M=\bigsqcup_{H<G}M_{(H)},$$
 makes $M$ into a Whitney stratified space, where $M_{(H)}$ is the set of points in $M$ such that the isotropy subgroup of each point is conjugate to $H$.
\end{ex}
\begin{rem}
In general, a stratum $M_{(H)}$ may have connected components with different dimensions.
In this case we can refine the decomposition to make each piece of the stratification is a submanifold in $M$.
To keep our notation manageable we refine such decomposition and still write it as $M_{(H)}$.
\end{rem}

\begin{defn}
Let $X$ be a closed subset of a smooth manifold $M$.
We say that $X$ admits a \emph{Whitney stratification}, if there exists a stratification $\mathcal{S}$ on $X$ with a filtration of $X$ by closed subsets
$$
X=X_{m}\supseteq X_{m-1}\supseteq\cdot\cdot\cdot\supseteq X_{0}\supseteq X_{-1}=\emptyset,
$$
and any pair of strata $(S_{i},S_{j}),(i\leq j)$ satisfies \emph{Condition B}.
\end{defn}
The subset $X$ together with the Whitney stratification is called a \emph{Whitney object}.
Specially, if $W\subset X$ is a closed subset with a Whitney stratification such that each stratum of $W$ is contained in a single stratum of $X$, then $W$ is called a \emph{Whitney substratified object} of $X$.

The idea of representing cocycles by geometric objects was introduced by Whitney in \cite{HW47}.
In \cite{RMG81}, Goresky gave all technical constructions for a geometric description of homology and cohomology in the context of Whitney stratifications.
Let us recall the Goresky's method of geometric chains (cycles).

\begin{defn}[Goresky \cite{RMG81}]
A \emph{geometric $k$-chain} $\xi$ in a fixed Whitney object $X$ consists of a compact $k$-dimensional Whitney substratified object $|\xi|\subset X$, which is called the support of $\xi$, together with an orientation of $|\xi|$ which is a choice of an orientation and multiplicity of each $k$-dimensional stratum.
The set of orientations of $|\xi|$ is just the group $H_{k}(|\xi|,|\xi|_{k-1})$.
\end{defn}

\begin{defn}
Let $\xi$ be a geometric $k$-chain in $X$.
The \emph{reduction} of $\xi$ is the geometric chain whose support is the closure of the union of all components of $|\xi|-|\xi|_{k-1}$ which have been assigned a nonzero multiplicity.
In particular, we can identify a geometric chain with its reduction.
\end{defn}

Consider the pairs $(|\xi|,|\xi|_{k-1})$ and $(|\xi|_{k-1},|\xi|_{k-2})$, there exist two exact homology sequences as follows:
\begin{equation}\label{1}
\xymatrix@C=0.5cm{
  \cdot\cdot\cdot \ar[r] & H_{k}(|\xi|_{k-1}) \ar[r]^{i_{*}} & H_{k}(|\xi|)\ar[r]^{j_{*}\quad\,\,}& H_{k}(|\xi|,|\xi|_{k-1}) \ar[r]^{\partial_{k}} & H_{k-1}(|\xi|_{k-1}) \ar[r] & \cdot\cdot\cdot }
\end{equation}
and
\begin{equation}
\xymatrix@C=0.5cm{
  \cdot\cdot\cdot \ar[r] & H_{k-1}(|\xi|_{k-2}) \ar[r]^{i_{*}} & H_{k-1}(|\xi|_{k-1}) \ar[r]^{j_{*}\quad\,} & H_{k-1}(|\xi|_{k-1},|\xi|_{k-2}) \ar[r]^{\quad\,\partial_{k-1}} & H_{k-2}(|\xi|_{k-2}) \ar[r] & \cdot\cdot\cdot }
\end{equation}
where $\partial_{k}$ and $\partial_{k-1}$ are the boundary operators.
Given a geometric $k$-chain $\xi$, the boundary of $\xi$, denoted by $\partial\xi$, is defined to be the geometric $(k-1)$-chain with the support $|\xi|_{k-1}$ and the orientation induced from sequence
\begin{equation}
\xymatrix@C=0.5cm{
   & H_{k}(|\xi|,|\xi|_{k-1}) \ar[r]^{\partial_{k}} & H_{k-1}(|\xi|_{k-1}) \ar[r]^{j_{*}\quad\,} & H_{k-1}(|\xi|_{k-1},|\xi|_{k-2}). }
\end{equation}

\begin{defn}
We say that $\xi$ is a \emph{geometric $k$-cycle} if the boundary of $\xi$ satisfies $\partial\xi=0$.
\end{defn}

Note that the orientation of $\xi$ is a homology class $O_{\xi}\in H_{k}(|\xi|,|\xi|_{k-1})$, and $\partial\xi=0$ implies that $\partial_{k}O_{\xi}=0$.
Due to the exactness of the homology sequence (\ref{1}) there exists a unique \emph{fundamental class} $\mu_{\xi}\in H_{k}(|\xi|)$ such that $j_{*}(\mu_{\xi})=O_{\xi}$. Let $\iota:|\xi|\longrightarrow X$ be the inclusion, then $\xi$ represents a homology class $[\xi]=\iota_{*}\mu_{\xi}\in H_{k}(X)$. There is an equivalence relation, called \emph{cobordism}, between geometric $k$-cycles.

\begin{defn}
Let $\xi_{0}$ and $\xi_{1}$ be two geometric $k$-cycles in $X$.
They are called \emph{cobordant} if there exists a geometric $(k+1)$-chain $\eta$ in $X\times \mathbb{R}$ and some $\varepsilon>0$ such that
\begin{enumerate}
  \item $|\eta|\subset X\times [0,1]$;
  \item $|\eta|\cap X\times[0,\varepsilon)=|\xi_{0}|\times[0,\varepsilon)$;
  \item $|\eta|\cap X\times(1-\varepsilon,1]=|\xi_{1}|\times(1-\varepsilon,1]$;
  \item $\partial\eta=\xi_{1}\times\{1\}-\xi_{0}\times\{0\}$ (modulo reduction).
\end{enumerate}
\end{defn}

Denote the set $WH_{k}(X)$ by the cobordism classes of geometric $k$-cycles in $X$.
Note that the cobordant cycles represent the same homology class, and therefore, we get a \emph{representation map}
\begin{equation}\label{2}
  R:WH_{k}(X)\longrightarrow H_{k}(X).
\end{equation}
In particular, assume that $Y\subset X$ is an oriented compact Whitney substratified object.
If there are no strata of $Y$ with codimension one, then the cycle condition automatically holds, i.e. $Y$ represents a homology class in $X$.

\subsection{General position of equivariant smooth maps}

In this subsection, we recall the definition of general position for a $G$-equivariant map and state some elementary properties,
for more details refer to \cite{EB77}.

Let $V$ be finite dimensional vector space.
Then $V$ is called a \emph{$G$-space} if there exists a representations of $G$ over $V$, $\rho_{V}:G\longrightarrow GL(V)$.
A smooth map $F:V\longrightarrow W$ of two $G$-spaces is called a $G$-equivariant map if for any $g\in G$ we have
$
(\rho_{W}(g))\circ F=F\circ(\rho_{V}(g)).
$
The set of all smooth $G$-equivariant maps is denoted by $\mathcal{C}^{\infty}_{G}(V,W)$.
Let $G$ acts on the $\mathbb{R}$ trivially.
A smooth function $f:V\longrightarrow \mathbb{R}$ is \emph{$G$-invariant} if it satisfy the condition $f\circ(\rho_{V}(g))=f$, for all $g\in G$.

Let $\mathcal{C}^{\infty}_{G}(V)$ be the set of $G$-invariant smooth functions on $V$.
Then $\mathcal{C}^{\infty}_{G}(V,W)$ has the structure of a $\mathcal{C}^{\infty}_{G}(V)$-module with finite polynomial generators (\cite[Lemma 3.1]{MF177}).
Suppose that $\{F_{_{1}},...,F_{k}\}$ is the set of polynomial generators, then for every $G$-equivariant map $F\in \mathcal{C}^{\infty}_{G}(V,W)$ there exist unique $G$-invariant functions $h_{i}\in \mathcal{C}^{\infty}_{G}(V)$,$(1\leq i\leq k)$ such that
$$
F(x)=\sum^{k}_{i=1}h_{i}(x)F_{i}(x) ,\forall x\in V.
$$
Define the map
\begin{equation}
U:V\times \mathbb{R}^{k}\longrightarrow W,\quad(x;t_{1},...,t_{k})\longmapsto \sum^{k}_{i=1}t_{i}F_{i}(x).
\end{equation}
The zero set of the $U$, denoted by
$$
\mathcal{E}:=\{(x,t)\in V\times \mathbb{R}^{k}\mid U(x,t)=0\},
$$
is called the \emph{universal variety}.

The universal variety $\mathcal{E}$ is a real affine algebraic variety which is uniquely determined (up to product with an affine space) by $V$ and $W$; moreover, $\mathcal{E}$ admits a unique minimum Whitney stratification.

\begin{defn}
 Define the map
\begin{equation}
\Gamma(F):V\longrightarrow V\times \mathbb{R}^{k},\quad x\longmapsto (x,h_{1}(x),...,h_{k}(x)).
\end{equation}
The map $\Gamma(F)$ is called the \emph{graph map} of $F$.
\end{defn}

Clearly we get $F=U\circ \Gamma(F)$ and $F^{-1}(0)=\Gamma(F)^{-1}(\mathcal{E})$.
The universal variety $\mathcal{E}$ contains the information about all possible zero sets for $F\in \mathcal{C}^{\infty}_{G}(V,W)$.
Suppose that $X$ is a smooth manifold and $E\subset\mathbb{R}^{q}$ is an algebraic subvariety.
A smooth map $f:X\longrightarrow \mathbb{R}^{q}$ is \emph{transverse to} $E$ means that $f$ is transverse to each stratum of the minimum Whitney stratification of $E$.

\begin{defn}\label{def2.2.2}
Let $F\in\mathcal{C}^{\infty}_{G}(V,W)$ such that $F(0)=0$.
Then $F$ is \emph{in general position with respect to} $0\in W$ at $0\in V$, if the graph map $\Gamma(F)$
is transversal to the minimum Whitney stratification of the universal variety $\mathcal{E}$ in $V\times \mathbb{R}^{k}$.
\end{defn}

\begin{defn}\label{3}
If $W=W_{1}\oplus W_{2}$ is a direct sum decomposition of $G$-spaces $W_{1}$ and $W_{2}$, then the $G$-equivariant map
$$F=(F_{1},F_{2}):V\longrightarrow W_{1}\oplus W_{2}$$
is in general position with respect to $W_{1}\subset W$ at $0\in V$ if and only if the map
$$F_{2}:V\longrightarrow W_{2}$$
is in general position with respect to $0\in W_{2}$ at $0\in V$.
\end{defn}

Next we review some basic properties of smooth actions of a compact Lie group on manifolds.
Let $G$ be a compact Lie group and $M$ be a smooth manifold.
A \emph{smooth $G$-action} on $M$ is a smooth map
\begin{equation}
l:G\times M\longrightarrow M,\,\,(g,x)\longmapsto gx
\end{equation}
such that $l(e,x)=x$ and $(g_{1}g_{2})x=g_{1}(g_{2}x)$ for any $ g_{1},g_{2}\in G,x\in M.$

For any $g\in G$ we can construct a smooth map
\begin{equation}
\mu_{g}:M\longrightarrow M,\,\, x\longmapsto gx.
\end{equation}
A point $x\in M$ is called a \emph{fixed point} if $\mu_{g}(x)=x$ for any $g\in G$. For any $x\in M$ the set
$$G(x)=\{gx\mid \forall g\in G\}$$
is called the \emph{orbit of $x$}, and the closed subgroup
$$G_{x}=\{g\in G\mid gx=x\}$$
is called the \emph{isotropy subgroup} of $x$.
The action of $G$ on $M$ is \emph{effective} if the map $\mu_{g}$ is the identity mapping on $M$ only for $g=e$, where $e$ is the identity of $G$.
We say that the $G$-action of $G$ is \emph{free} if for any $x\in M$, $\mu_{g}(x)= x$ implies that $g=e$.

\begin{defn}
Assume that the Lie group $G$ acts smoothly on manifold $M$.
The action is \emph{proper} if for every compact subset $K\subset M$, the set
$$
G_{K}=\{g\in G\mid (gK)\cap K\neq \emptyset\}
$$
is compact.
\end{defn}

If $G$ is a compact Lie group, then the smooth action of $G$ on a smooth manifold $M$ is proper.
A manifold $M$ is called a \emph{$G$-manifold} if $G$ acts on $M$ smoothly.
In particular, if the action is proper then $M$ is called a \emph{proper $G$-manifold}.
The set of all closed subgroups of $G$ admits an equivalence relation as follows:
 $$H\sim H^{\prime}\Longleftrightarrow H=gH^{\prime}g^{-1}$$
for some $g\in G$.
The equivalence classes, denoted by $(H)$, are called the \emph{conjugacy classes}.
Moreover, the set of conjugacy classes bears a partial order:
 $(H)\leq(H^{\prime})$ if there exists a $g\in G$ such that $H\subseteq gH^{\prime}g^{-1}$.
For an orbit $G(x)$ the isotropy groups $G_{gx}$ form a conjugacy class $(G_{x})$, which is called the \emph{isotropy type} of the orbit $G(x)$.

Given any $x\in M$, the orbit $G(x)$ describes a $G$-invariant closed submanifold of $M$.
Furthermore, $G(x)$ is isomorphic to $G/G_{x}$ via the canonical mapping
 \begin{equation}
 \Phi_{x}:G/G_{x}\longrightarrow M,\,\, gG_{x}\longmapsto gx.
\end{equation}

In particular, for a proper $G$-manifold there exists the \emph{$G$-invariant partition of unity} (cf. \cite[Theorem 4.2.4.]{JP01}) for any covering of the manifold by $G$-invariant open subsets and \emph{$G$-invariant tubular neighborhood theorem} (cf. \cite[Theorem 2.2.]{GB72}) for a closed invariant submanifold.

For each $x\in M$, the normal space $S_{x}=T_{x}M/T_{x}(G(x))$ is called the slice of $x$.
Note that the homogenous space $G\longrightarrow G/G_{x}$ is a $G_{x}$-principal fiber bundle; furthermore, we get an associated bundle $N_{x}=G\times_{G_{x}}S_{x}$, which is called the \emph{slice bundle} of $x$.
The following \emph{differential slice theorem} shows that every smooth $G$-manifold $M$ with a proper $G$-action locally looks like a neighborhood of the zero section in the slice bundle.
For the original proof of the slice theorem refer to Palais \cite{RP61}.

 \begin{thm}[cf. \cite{JP01}, Theorem 4.2.6]
Let $\Phi:G\times M\longrightarrow M$ be a proper group action, $x$ a point of $M$ and $S_{x}=T_{x}M/T_{x}(G(x))$ the normal space to the orbit of $x$.
Then there exists a $G$-equivariant diffeomorphism from a $G$-invariant neighborhood of the zero section of $G\times_{G_{x}}S_{x}$ onto a $G$-invariant neighborhood of $G(x)$ such that the zero section is mapped onto $G(x)$ in a canonical way.
\end{thm}

Summarize what we can get from the differential slice theorem: If $G$ is a compact Lie group, then for each $x\in M$ there exists a $G_{x}$-invariant submanifold in $M$, denoted by $S$, such that
\begin{enumerate}
 \item $x\in S$;
 \item $gS\cap S\neq \emptyset\Longrightarrow g\in G_{x}$;
 \item $\forall y\in S\Longrightarrow G_{y}\subseteq G_{x}$;
 \item $GS=\{gy\mid g\in G,y\in S\}$ is a $G$-invariant open neighborhood of the orbit $G(x)$ in $M$.
 \end{enumerate}

\begin{defn}
Suppose that $M$ and $N$ are smooth $G$-manifolds.
A smooth map $f\in \mathcal{C}^{\infty}(M,N)$ is called a \emph{$G$-equivariant smooth map} if the $G$-action commutes with $f$, i.e.
$$f(gx)=g(f(x))$$
for any $x\in M$ and $g\in G$.
\end{defn}
Clearly, if $f$ is $G$-equivariant then $G_{x}\subseteq G_{f(x)}$ for any $x\in M$.
Let $P\subset N$ be a $G$-invariant submanifold.
For any $x\in f^{-1}(P)$, we can choose a $G_{x}$-equivariant diffeomorphism $\phi$ from a neighborhood $V$ of $f(x)$ in $N$ to a $G_{x}$-vector space $W_{1}\oplus W_{2}$ such that $\phi(V\cap P)=W_{1}$.
The pair $(V,\phi)$ is called a $G_{x}$-chart for $P$ at $f(x)$.
Choose a slice at $x$, denoted by $S_{x}$, then $f$ determines a smooth $G_{x}$-equivariant map
$$\tilde{f}:S_{x}\longrightarrow W_{1}\oplus W_{2}.$$
\begin{defn}
We say that $f$ is \emph{in general position} with respect to $P$ at $x\in M$, if
\begin{enumerate}
  \item $f(x)\notin P$; or
  \item $f(x)\in P$ and for some choice of slice $S_{x}$ and $G_{x}$-chart for $P$ at $f(x)$
  $$\tilde{f}:S_{x}\longrightarrow W_{1}\oplus W_{2}$$
  is in general position with respect to $W_{1}$ at $x$ in the sense of Definition \ref{3}.
\end{enumerate}

If $f$ is in general position with respect to $P$ at every point of $M$, then we say that $f$ is in general position with respect to $P$.
\end{defn}
\begin{rem}
 The above definition is well-defined, since the definition is independent of the choice of the slices (cf. \cite[ Proposition 5.6.]{EB77}).
\end{rem}
If $f$ is in general position with respect to $P$ at $x\in M$, then the definition implies that $f$ is in general position with respect to $P$ at $gx$ for any $g\in G$; furthermore, it is in general position in a neighborhood of $x$.

In classical differential topology the set of smooth maps which are in general position is open and dense.
Similarly, for the $G$-manifolds and equivariant maps we have:
\begin{thm}[Bierstone \cite{EB77}, Theorem 1.4]\label{thm2.2.9}
Suppose that $P$ is an invariant submanifold of $N$, then the set of smooth equivariant maps $f \in \mathcal{C}^{\infty}_{G}(M,N)$ which are in general position with respect to $P$ is a countable intersection of open dense sets in the Whitney of $\mathcal{C}^{\infty}$-topology.
\end{thm}
Particularly, when $P$ is a closed invariant submanifold, the set of smooth equivariant maps which are in general position admit the openness and density in the Whitney topology.
Furthermore, Bierstone showed that if $f\in\mathcal{C}^{\infty}_{G}(M,N)$ is in general position with respect to $P$, then $f^{-1}(P)\subset M$ is a Whitney object and every stratum of $f^{-1}(P)$ is a $G$-invariant submanifold of $M$.
Note that the proper action of a Lie group on a manifold induces a natural stratification, thus, a natural approach to the transversality problem in equivariant case is to consider the stratumwise transversality of a smooth equivariant map $f:M\longrightarrow N$ with respect to an invariant submanifold $P$ of $N$.
Given a closed subgroup $H$ of $G$ we can assign the following three subspaces of $M$
$$M_{(H)}:=\{x\in M\mid G_{x}\sim H\};$$
$$M^{H}:=\{x\in M\mid G_{x}\supset H\};$$
$$M_{H}:=\{x\in M\mid G_{x}= H\}.$$
In general, $M_{(H)}$, $M^{H}$ and $M_{H}$ are not connected and each connected component is a submanifold of $M$.
In fact, $M^{H}$ is nothing else but the fixed points of $H$ in $M$.
Especially, if $H\subset G$ is compact then these three spaces satisfy the following relation
$$M_{H}=M_{(H)}\cap M^{H}.$$ Since $G$ is a compact Lie group and $H$, being an isotropy subgroup of some points in $M$ is a closed subgroup of $G$, the compactness of $H$ is automatic.
Furthermore, $M_{(H)}$ can be decomposed into the form of
$$M_{(H)}=\bigsqcup_{J\in (H)}M_{J}.$$

The equivariant map $f$ does not always map $M_{(H)}$ into the set $N_{(H)}$ since in general $G_{x}\subseteq G_{f(x)}$, however, $f$ maps $M^{H}$ into $N^{H}$. Let $f_{H}=f\mid_{M_{H}}$, then the image of $f_{H}$ is contained in $N^{H}$ and the stratumwise transversality of $f$ is defined as follows.
\begin{defn}
Let $f\in \mathcal{C}^{\infty}_{G}(M,N)$ and $P\subseteq N$ be a $G$-invariant submanifold.
If for any subgroup $H$ of $G$ the map $f_{H}:M_{H}\longrightarrow N^{H}$ is transverse to the submanifold $P^{H}$ of $N^{H}$, then we say it is a $G$-equivariant map which admits the \emph{stratumwise transversality} with respect to the $G-$invariant submanifold $P$ of $N$.
\end{defn}

The concepts of general position and stratumwise transversality are generalizations of classical transversality in differential topology.
However, the stratumwise transversality is not open, i.e. if $f\in \mathcal{C}^{\infty}_{G}(M,N)$ admits the stratumwise transversality with respect to a $G$-invariant submanifold $P$ of $N$, then a small perturbation of $f$ may break the stratumwise transversality.
In particular, if a map $f\in \mathcal{C}^{\infty}_{G}(M,N)$ is in the general position with respect to a $G$-invariant submanifold $P$ of $N$, then $f$ is stratumwisely transverse to $P$ (cf. \cite[Proposition 6.4.]{EB77}).
In conclusion, we have the following relation for equivariant smooth maps
$$
\{\textmd{classical}\,\, \textmd{transversality}\}\subset\{\textmd{being}\,\, \textmd{in}\,\, \textmd{general}\,\,\textmd{ position}\}\subset\{\textmd{stratumwise}\,\, \textmd{transversality}\}.
$$
\section{Stratified obstruction system of $G$-moduli problem}
In this section we show that for an equivariant vector bundle there exists a family of obstruction bundles and we call it the stratified obstruction system.
The existence of the obstruction system implies that there is no equivariant perturbed section which is transversal to the zero section in general.

Let $G$ be a compact Lie group, and $B$ be a $G$-manifold.
A $G$-vector bundle over $B$ is defined as follows.
\begin{defn}\label{def3.1}
If $B$ is a $G$-manifold, a \emph{$G$-vector bundle} on $B$ is a $G$-space $E$ together with a $G$-equivariant map $\pi:E\rightarrow B$ such that
\begin{enumerate}
  \item $\pi:E\rightarrow B$ is a real vector bundle on $B$;
  \item for any $g\in G$, and $x\in B$ the action $g:E_{x}\rightarrow E_{gx}$ is homomorphism of vector space.
\end{enumerate}
\end{defn}
The above definition implies that
$$g^{-1}:E_{gx}\longrightarrow E_{x}$$
is also a homomorphism of vector spaces such that
$g^{-1}\circ g$ is the identity map on $E_{x}$ and $g\circ g^{-1}$ is the identity map on $E_{gx}$.
Thus $g:E_{x}\longrightarrow E_{gx}$ is an isomorphism of vector spaces.
Two $G$-vector bundles over $B$ are called \emph{$G$-equivalent} if they are equivalent as ordinary vector bundles via a $G$-equivariant bundle map.
Next we describe the appropriate generalization of product bundle in the equivariant case, which gives the local model of $G$-vector bundles.

Assume that $H$ is a closed subgroup of $G$ and $\rho:H\longrightarrow GL(\mathbb{R};k)$ a homomorphism.
For each $H$-space $V$ we denote by $\varepsilon^{\rho}(V)$ the $G$-vector bundle over $G\times_{H}V$ with fibre $\mathbb{R}^{k}$ given by
\begin{equation}
\pi:G\times_{H}(V\times\mathbb{R}^{k})\longrightarrow G\times_{H}V,\quad\pi([g,(v,e)])=[g,v]
\end{equation}
where $H$ acts on the fibre $\mathbb{R}^{k}$ via the homomorphism $\rho$.
Generally, for any $G$-space $X$ suppose that $H$ is a closed subgroup of $G$ and $V\subset X$ is a $H$-invariant subspace, then $V$ is called a \emph{$H$-slice} provided that the equivariant map
\begin{equation}
\mu:G\times_{H}V\longrightarrow X,\qquad\mu([g,v])=gv
\end{equation}
is a homomorphism onto an open subset of $X$.
\begin{defn}[Lashof \cite{RKL82}]
A $G$-vector bundle $\pi:E\longrightarrow B$ of rank $k$ is called \emph{$G$-locally trivial} if there exists a $G$-invariant open cover $\{GV_{\alpha}\}_{\alpha\in I}$ of $B$, where $V_{\alpha}$ is an $H_{\alpha}$-slice, such that the restriction $E\mid_{GV_{\alpha}}$ is $G$-equivalent to $\varepsilon^{\rho_{\alpha}}(V_{\alpha})$ for some homomorphism $\rho_{\alpha}:H_{\alpha}\longrightarrow GL(\mathbb{R};k)$ (under the identification $\mu:G\times_{H}V_{\alpha}\longrightarrow GV_{\alpha}$).
\end{defn}
In particular, every smooth $G$-vector bundle is $G$-locally trivial (cf. \cite[Corollary 1.6]{RKL82}).

According to Definition \ref{def3.1} (2), for any $x\in B$ there is a representation of the isotropy subgroup $G_{x}$ over the fibre $E_{x}$, therefore, $E_{x}$ is a $G_{x}$-vector space.
Denote by $E^{f}_{x}$ the $G_{x}$-fixed subspace of the fiber space $E_{x}$, i.e.
$$E^{f}_{x}=\{v\in E_{x}\mid gv=v, \forall g\in G_{x}\}.$$
Given a closed subgroup $H$ of $G$, let
$$B_{(H)}=\{x\in B| G_{x}\sim H\},$$
then we have:
\begin{prop}\label{3.0.13}
If $x$ and $y$ are contained in the same connected component of $B_{(H)}$, then $\emph{dim}\,E^{f}_{x}= \emph{dim}\,E^{f}_{y}$.
\end{prop}
To prove the proposition, we need the following lemma (cf. \cite[Lemma 6.12.]{PWM08}):
\begin{lem}\label{3.0.14}
Let $G$ be a compact Lie group, and $H$ a closed subgroup of $G$, then $gHg^{-1}\subset H$ implies that $gHg^{-1}=H$.
\end{lem}
\begin{proof}[Proof of Proposition \ref{3.0.13}]
Without loss of generality, we may assume that $B_{(H)}$ is connected.
Observe that $E|_{B_{(H)}}$ is a smooth $G$-vector bundle, thus it admits the $G$-local trivialization.
Since $B_{(H)}$ is connected given any smooth path
\footnote{Recall that a connected space is not always pathwise connected.
However, if a space is connected and locally pathwise connected then it is pathwise connected.
A connected manifold is always pathwise connected since the manifold is locally pathwise connected.}
$$\gamma:[0,1]\longrightarrow B_{(H)}$$
such that $\gamma(0)=x$ and $\gamma(1)=y$, we can find finitely many points on the path $\gamma$
$$x_{0}=x,x_{1},\cdot\cdot\cdot,x_{m}=y$$
with associated slices $V_{i}$ $(0\leq i\leq m)$ such that the $G$-invariant open subsets $\{GV_{i}\}^{m}_{i=0}$ cover the path $\gamma$, and $GV_{i}\cap GV_{i+1}\neq\emptyset$ for $0\leq i\leq m-1$.
In further, we can choose $GV_{i}$ such that for each $GV_{i}$, $\pi^{-1}(GV_{i})$ is $G$-equivalent to
$$\varepsilon^{\rho_{i}}(V_{i})=G\times_{H_{i}}(V_{i}\times \mathbb{R}^{k}),$$
where $H_{i}=G_{x_{i}}$ and $\rho_{i}:H_{i}\longrightarrow GL(\mathbb{R};k)$ is a homomorphism.
The next thing to do in the proof is to verify that for any $z\in GV_{i}$, we have $\textmd{dim}\,E^{f}_{z}=\textmd{dim}\,E^{f}_{x_{i}}$.

We divide the proof into three cases.

\paragraph{\textbf{Case 1}} Assume that $z\in V_{i}$, then $G_{z}\subset H_{i}$.
According to the definition of $B_{(H)}$, we get that $G_{z}$ and $G_{x_{i}}=H_{i}$ are conjugate to $H$. Since the conjugate relation is an equivalence relation, by the transitivity $G_{z}$ is conjugate to $H_{i}$, i.e. there exists a $g\in G$ such that $G_{z}=gH_{i}g^{-1}\subset H_{i}$.
By Lemma \ref{3.0.14}, we get $G_{z}=H_{i}$.
Furthermore, the representation of $G_{z}$ over the fibre $E_{z}$ is equivalent to $\rho_{i}|_{G_{z}}=\rho_{i}$.
It follows that the fixed subspaces $E^{f}_{z}$ and $E^{f}_{x_{i}}$ have the same dimension, i.e. $\textmd{dim}\,E^{f}_{z}=\textmd{dim}\,E^{f}_{x_{i}}$.

\paragraph{\textbf{Case 2}} Assume that $z\in G(x_{i})$, i.e. there exists a $g\in G$ such that $z=gx_{i}$.
Note that $z=gx_{i}$ implies that $G_{z}=gG_{x_{i}}g^{-1}$ and the map
\begin{equation}\label{4}
g:E_{x_{i}}\longrightarrow E_{z}
\end{equation}
is an isomorphism of vector spaces.
By definition we have:
$$E^{f}_{x_{i}}=\biggl\{u\in E_{x_{i}}\mid hu=u, \forall h\in G_{x_{i}}\biggr\}$$
and
\begin{eqnarray*}
E^{f}_{z}&=&\biggl\{v\in E_{z}\mid \tilde{g}v=v, \forall \tilde{g}\in G_{z}\biggr\}\\
&=&\biggl\{v\in E_{z}\mid \tilde{g}v=v, \forall \tilde{g}\in gG_{x_{i}}g^{-1}\biggr\}\quad\quad(G_{z}=gG_{x_{i}}g^{-1})\\
&=&\biggl\{v\in E_{z}\mid (ghg^{-1})v=v, \forall h\in G_{x_{i}}\biggr\}.
\end{eqnarray*}
For any $v\in E^{f}_{z}$, we have $g^{-1}v\in E_{x_{i}}$ and $h(g^{-1}v)=g^{-1}v$, for all $h\in G_{x_{i}}$.
It follows that the homomorphism of vector spaces
\begin{equation}
g^{-1}:E_{z}\longrightarrow E_{x_{i}}
\end{equation}
maps $E^{f}_{z}$ onto the subspace $E^{f}_{x_{i}}$.
It is straightforward to show that the homomorphism (\ref{4})
maps $E^{f}_{x_{i}}$ onto the subspace $E^{f}_{z}$.
So $E^{f}_{x_{i}}\cong E^{f}_{z}$ and we obtain $\textmd{dim}\,E^{f}_{z}= \textmd{dim}\,E^{f}_{x_{i}}$.

\paragraph{\textbf{Case 3}} Assume that $z\in GV_{i}$ and $z\notin V_{i}$.
Then there exists a $v\in V_{i}$ such that $z=gv$ for some $g\in G$.
On one hand, due to the result of Case 1 we have
\begin{equation}\label{5}
\textmd{dim}\,E^{f}_{v}=\textmd{dim}\,E^{f}_{x_{i}}.
\end{equation}
On the other hand from the result of Case 2 we get
\begin{equation}\label{6}
\textmd{dim}\,E^{f}_{v}=\textmd{dim}\,E^{f}_{z}.
\end{equation}
At last, according to (\ref{5}) and (\ref{6}) we obtain
\begin{equation}
\textmd{dim}\,E^{f}_{z}=\textmd{dim}\,E^{f}_{x_{i}}.
\end{equation}

Note that the path $\gamma$ is covered by the $G$-invariant subsets $\{GV_{i}\}^{m}_{i=0}$ and for each $0\leq i\leq m-1$ $GV_{i}\cap GV_{i+1}\neq \emptyset$, and therefore, we have
$$\textmd{dim}\,E^{f}_{x}=\textmd{dim}\,E^{f}_{x_{1}}=\cdot\cdot\cdot=\textmd{dim}\,E^{f}_{x_{m-1}}=\textmd{dim}\,E^{f}_{y}.$$
This completes the proof.
\end{proof}
\begin{defn}
(Finite dimensional $G$-moduli problem) Let $G$ be an oriented compact Lie group.
A \emph{finite dimensional $G$-moduli problem} is a triple $(B,E,S)$ with the following properties:
\begin{itemize}
  \item $B$ is a compact smooth manifold (without boundary) on which $G$ acts smoothly;
  \item $E$ is a $G$-vector bundle over $B$;
  \item $S:B\longrightarrow E$ is a $G-$equivariant smooth section.
\end{itemize}
A $G$-moduli problem $(B,E,S)$ is \emph{oriented} if $B$ and $E$ are oriented and $G$ acts on $B$ and $E$ by orientation preserving diffeomorphisms.
\end{defn}
In order to make our notation manageable, from now on we assume that $B_{(H)}$ is connected for each closed subgroup $H$ of $G$.
In the general case we may consider it component by component.
From Proposition \ref{3.0.13}, for any $x,y\in B_{(H)}$ we have $\textmd{dim}\,E^{f}_{x}=\textmd{dim}\,E^{f}_{y}$, i.e. the dimension of $E^{f}_{x}$ is independent of the choice of $x\in B_{(H)}$. Assume that $\textmd{rank}\,E=k$ and $\textmd{dim}\,E^{f}_{x}=l$, let
$$\mathcal{F}=\{E^{f}_{x}\}_{x\in B_{(H)}},$$
then the collection $\mathcal{F}$ is an $G$-invariant subspace of $E\mid_{B_{(H)}}$.
Using the local trivialization of $E\mid_{B_{(H)}}$, we get that for every $x\in B_{(H)}$, there exists a neighborhood $U$ of $x$ in $B_{(H)}$ and a trivialization
\begin{equation}
\phi_{U}:E|_{U}\longrightarrow U\times \mathbb{R}^{k}.
\end{equation}
Moreover, the restriction of $\phi_{U}$ on $\mathcal{F}\mid_{U}$ induces a map
\begin{equation}
\tilde{\phi}_{U}:\mathcal{F}|_{U}\longrightarrow U\times R^{l}\subset U\times \mathbb{R}^{k},
\end{equation}
which gives a local trivialization of $\mathcal{F}$ over $U$.
For any pair of trivializations $\phi_{U}$ and $\phi_{V}$ of $E\mid_{B_{(H)}}$ we have the smooth transition functions
\begin{equation}
g_{UV}:U\cap V\longrightarrow GL(\mathbb{R};k)
\end{equation}
given by
$$g_{UV}(x)=(\phi_{U}\circ\phi^{-1}_{V})|_{\{x\}\times \mathbb{R}^{k}}$$
and satisfying the cocycle condition:
\begin{equation}\label{7}
g_{UV}\cdot g_{VW}=g_{UW}\quad(U\cap V\cap W\neq\emptyset).
\end{equation}
For any $x\in U\cap V$ the fibre $E_{x}$ is a $G_{x}$-vector space.
For each $e\in E_{x}$ under the trivialization $\phi_{U}$ we get $\phi_{U}(e)\in \mathbb{R}^{k}$, and similarly, under the trivialization $\phi_{V}$, we have $\phi_{V}(e)\in \mathbb{R}^{k}$.
Consider the action of $g\in G_{x}$ on $e$.
Under the different trivializations $\phi_{U}$ and $\phi_{V}$ we get:
\begin{equation}\label{8}
  \phi^{-1}_{U}\cdot T_{U}\cdot \phi_{U}(e)=ge,
\end{equation}
\begin{equation}\label{9}
  \phi^{-1}_{V}\cdot T_{V}\cdot \phi_{V}(e)=ge,
\end{equation}
where $T_{U}:=(\phi_{U}\cdot g\cdot\phi^{-1}_{U})|_{x}$ and $T_{V}:=(\phi_{V}\cdot g\cdot\phi^{-1}_{V})|_{x}$.

According to (\ref{8}) and (\ref{9}), we get
\begin{equation}\label{10}
  \phi_{U}\circ\phi^{-1}_{V}\cdot T_{V}\cdot \phi_{V}(e)=T_{U}\cdot \phi_{U}(e).
\end{equation}
Since $g_{UV}=\phi_{U}\circ\phi^{-1}_{V}$ and $\phi_{U}(e)=g_{UV}(x)\circ\phi_{V}(e)$, (\ref{10}) is equivalent to
\begin{equation}
  g_{UV}(x)\circ T_{V}\circ\phi_{V}(e)=T_{U}\circ g_{UV}(x)\circ\phi_{V}(e).
\end{equation}
This implies that the action of $G_{x}$ on $E_{x}$ is independent of the trivialization.
Thus $g_{UV}$ induces a transition function
\begin{equation}
\tilde{g}_{UV}:U\cap V\longrightarrow GL(\mathbb{R};l)
\end{equation}
given by
$$\tilde{g}_{UV}(x)=(\tilde{\phi}_{U}\circ\tilde{\phi}^{-1}_{V})|_{\{x\}\times \mathbb{R}^{l}}.$$
The cocycle condition of transition functions $\{\tilde{g}_{UV}\}$ is determined by (\ref{7}).
This implies that the collection $\mathcal{F}$ with the cocycle $\{\tilde{g}_{UV}\}$ forms a subbundle of $E\mid_{B_{(H)}}$, denoted by $\mathcal{E}_{H}$.
Since the projection $\pi:\mathcal{E}_{H}\longrightarrow B_{(H)}$ is $G$-equivariant and for any $g\in G$ and $x\in B_{(H)}$ the action
\begin{equation}
g:E^{f}_{x}\longrightarrow E^{f}_{gx}
\end{equation}
is a homomorphism of vector spaces, $\mathcal{E}_{H}$ is a $G$-equivariant subbundle.
Observe that the section $S:B\longrightarrow E$ is equivariant, for any $x\in B$ and $g\in G_{x}$ we have
$$
g(S(x))=S(gx)=S(x).
$$
This follows that $S(x)$ is contained in the $G_{x}$-fixed subspace $E^{f}_{x}\subset E_{x}$.
Therefore $S_{H}:=S\mid_{B_{(H)}}$ is an equivariant smooth section of $\mathcal{E}_{H}$; moreover we say that the triple $(B_{(H)},\mathcal{E}_{H},S_{H})$ is the \emph{fixed subbundle} of $E\mid_{B_{(H)}}$ with the induced $G$-equivariant smooth section.
\begin{defn}
(Partition of $G$-moduli problem) The family of the fixed subbundles
$$\{(B_{(H)},\mathcal{E}_{H},S_{H})\mid H<G\}$$
 is called the \emph{partition of $(B,E,S)$}. Also we write
 $$(B,E,S)=\bigsqcup_{(H)}(B_{(H)},\mathcal{E}_{H},S_{H})$$
 where $(H)$ runs over the all isotropy classes.
\end{defn}
Define $\mathcal{O}_{H}$ be the quotient bundle of $\mathcal{E}_{H}$, then it is also a $G$-equivariant bundle and we have a direct sum decomposition of $G$-vector bundles
$$E\mid_{B_{(H)}}=\mathcal{E}_{H}\oplus \mathcal{O}_{H}.$$
In particular, we say that the vector bundle $\mathfrak{o}_{H}:\mathcal{O}_{H}\longrightarrow B_{(H)}$
is the \emph{obstruction bundle} of $E\mid_{B_{(H)}}$.
\begin{defn}
(Obstruction system of $G$-moduli problem) The family of the obstruction bundles
$$\{(B_{(H)},\mathcal{O}_{H},\mathfrak{o}_{H})\mid H<G\}$$
is called the \emph{obstruction system} of $(B,E,S)$.
\end{defn}
\begin{defn}
(Coindex of $G$-vector bundle) The \emph{coindex} of $G$-vector bundle $\pi:E\longrightarrow B$ is defined to be the integer
$$\textmd{coind}(B,E)=\max_{H<G, H\neq e}\{\textmd{codim}B_{(H)}-\textmd{rank}\mathcal{O}_{H}\},$$
where $e$ is the identity of $G$. The coindex is uniquely determined by the $G$-actions on $B$ and $E$.
\end{defn}
Suppose that the $G$-moduli problem $(B,E,S)$ is oriented, then the orientation on $E$ determines an orientation on each fibre $E_{x}$.
In particular, the induced orientation on $E_{x}$ is preserved by the $G_{x}$-action on $E_{x}$.
Let $E^{m}_{x}=E_{x}/E^{f}_{x}$ be the quotient space of $E^{f}_{x}$, i.e. the moving part under the $G_{x}$-action, then each fibre $E_{x}$ can be decomposed into the direct sum of $G_{x}$-subspaces as follows
\begin{equation}
E_{x}=E^{f}_{x}\oplus E^{m}_{x}.
\end{equation}
In fact, the moving subspace $E^{m}_{x}$ is the fibre of obstruction bundle $\mathcal{O}_{H}$ at $x\in B_{(H)}$.
The orientation on $E_{x}$ induce the orientations on $E^{f}_{x}$ and $E^{m}_{x}$ which are preserved by the $G_{x}$-action.
Therefore we assign to each fibre of fixed subbundle $\mathcal{E}_{H}$ an induced orientation.
If this induced orientation is smooth, then $\mathcal{E}_{H}$ is \emph{oriented}.
In particular, the obstruction bundle $$\mathfrak{o}_{H}:\mathcal{O}_{H}\longrightarrow B_{(H)}$$
is \emph{oriented} if and only if $\mathcal{E}_{H}$ is oriented.
We say that the obstruction system of an oriented $G$-moduli problem is \emph{oriented}, if each obstruction bundle is oriented.
Next we consider the transversality of the partition of $G$-moduli problem.
We say that the partition $\{(B_{(H)},\mathcal{E}_{H},S_{H})\mid H<G\}$ is \emph{transversal}, if for each $H$ of $G$ the section $S_{H}$ is transverse to the zero section of $\mathcal{E}_{H}$.

\begin{prop}\label{3.0.19}
Let $(B,E,S)$ be a $G$-moduli problem.
If $S:B\longrightarrow E$ is in general position with respect to the zero section over $B$, then the partition
$$\{(B_{(H)},\mathcal{E}_{H},S_{H})\mid H<G\}$$
is transversal.
\end{prop}
\begin{proof} Notice that $E$ is a $G$-vector bundle over $B$, and therefore $B$ can be considered as an embedded $G$-invariant closed submanifold in $E$.
Given any closed subgroup $H$ of $G$, let $J\in (H)$.
For any $b\in B$ assume that $J\subset G_{b}$ and define $(E_{b})^{J}$ as the $J$-fixed subspace of $E_{b}$, i.e.
$$(E_{b})^{J}=\{e\in E_{b}| ge=e,\forall g\in J\}.$$
Consider the $J$-fixed submanifold of $E$:
\begin{eqnarray*}
E^{J}&=&\{(b,e)\in E| g(b,e)=(b,e),\forall g\in J\}\\
&=&\{(b,e)\in E| gb=b,ge=e,\forall g\in J\}\\
&=&\{(b,e)\in E| b\in B^{J},e\in (E_{b})^{J}\}\\
&=&\bigsqcup_{b\in B^{J}}(E_{b})^{J}.
\end{eqnarray*}
Observe that the section $S:B\longrightarrow E$ is in general position with respect to the zero section over $B$, and hence it admits the stratumwise transversality, i.e.
\begin{equation}
S|_{B_{J}}:B_{J}\longrightarrow E^{J}
\end{equation}
is transverse to the $J$-fixed submanifold $B^{J}\subset E^{J}$.
For each $b\in S^{-1}(0)\cap B_{J}$ we have $G_{b}=J$ and
\begin{eqnarray*}
dS(b)(T_{b}(B_{J}))+T_{b}(B^{J})&=&T_{b}(E^{J})\\
&=&T_{b}(B^{J})\oplus (E_{b})^{J}\\
&=&T_{b}(B^{J})\oplus (E_{b})^{G_{b}}\quad(G_{b}=J)\\
&=&T_{b}(B^{J})\oplus E^{f}_{b}.\quad\quad\,\,((E_{b})^{G_{b}}=E^{f}_{b})
\end{eqnarray*}
It follows that the linear map

\begin{equation}\label{11}
\xymatrix@C=0.5cm{
    & \delta S(b):T_{b}(B_{J}) \ar[rr]^{dS(b)} && T_{b}(B^{J})\oplus E^{f}_{b} \ar[rr]^{\quad\,\:\: proj} && E^{f}_{b}  &  }
\end{equation}
is surjective.
Consider the $H$-fixed subbundle $(B_{(H)},\mathcal{E}_{H},S_{H})$.
In order to verify that the section
\begin{equation}
S_{H}:B_{(H)}\longrightarrow \mathcal{E}_{H}
\end{equation}
is transversal to the zero section, we only need to show that for each $b\in S^{-1}(0)\cap B_{(H)}$ the vertical differential
\begin{equation}
\xymatrix@C=0.5cm{
    & \delta S(b):T_{b}B_{(H)} \ar[rr]^{\quad dS(b)} && T_{b}(\mathcal{E}_{H}) \ar[rr]^{\quad\, proj} && E^{f}_{b}  &  }
\end{equation}
is surjective.
Note that $B_{(H)}=\sqcup_{J\in (H)}B_{J}$, for any $b\in S^{-1}(0)\cap B_{(H)}$ there exists a $B_{J}$ such that $b\in S^{-1}(0)\cap B_{J}$.
The stratumwise transversality of $S$ implies that the map (\ref{11})
is surjective.
Observe that $T_{b}(B_{J})$ is a tangent subspace of $T_{b}B_{(H)}$, and therefore the vertical differential
$$\delta S(b):T_{b}B_{(H)}\longrightarrow E^{f}_{b}$$
is surjective.
It follows that for any $H\subset G$ the section
$$S_{H}:B_{(H)}\longrightarrow \mathcal{E}_{H}$$
is transverse to the zero section; moreover, the partition
$$\{(B_{(H)},\mathcal{E}_{H},S_{H})\mid H<G\}$$ is transversal.
\end{proof}

\section{Invariant Euler cycle of $G$-moduli problem }
In this section we give the proof of Theorem \ref{Main Thm1}.

\begin{proof}[Proof of Theorem \ref{Main Thm1}]
Our first goal is to show that there exists a $G$-equivariant perturbation $P:B\longrightarrow E$ supported in a $G$-invariant open neighborhood of $S^{-1}(0)$ such that $S+P$ is in general position with respect to the zero section over $B$.
The idea of the proof is canonical.
We can reduce the problem to the local situation and construct a local equivariant perturbation.
Then using the $G$-invariant partition of unity we can glue those local perturbations to get a global one.

For any $x\in S^{-1}(0)$, let $S_{x}$ be the slice at $x$.
From the differential slice theorem, there exists a triple $(U,\phi,G_{x})$ satisfying the following properties:
\begin{enumerate}
  \item $U\subset S_{x}$ is a $G_{x}$-invariant open neighborhood of zero in the $G_{x}$-vector space $S_{x}$.
  \item $\phi:U\longrightarrow B$ is a $G_{x}$-equivariant embedding such that $x=\phi(0)$.
  \item $\phi$ induces a $G$-equivariant diffeomorphism form $G\times_{G_{x}}U$ onto a $G$-invariant open neighborhood of $x$ in $B$, denoted by $W$, as follows
  $$\Phi:[g,y]\longmapsto g\cdot\phi(y)$$
   where $[g,y]\in G\times_{G_{x}}U$ is the equivalence class determined by equivalence relation
   $$(g,y)\sim(h^{-1}\cdot g,h\cdot y),\,\forall h\in G_{x}.$$
\end{enumerate}
Since $S^{-1}(0)\subset B$ is compact we can choose finitely many points $x_{i}\in S^{-1}(0)\,(0\leq i\leq q)$ with triples $(U_{i},\phi_{i},H_{i})$ and induced maps $\Phi_{i}$ such that
$$S^{-1}(0)\subset \bigcup^{q}_{i=0}W_{i},$$
where $H_{i}$ is the isotropy subgroup of $x_{i}$ and $W_{i}$ is the $G$-invariant open neighborhood of $x_{i}$ in $B$ determined by the image of $\Phi_{i}$. Assume that $E_{i}$ is the fibre of $E$ at $x_{i}$. Since $U_{i}$ is a contractible neighborhood of zero in $S_{i}$ there exists an $H_{i}$-equivariant trivialization of the pullback bundle $\phi^{*}_{i}E=U_{i}\times E_{i}$.

Given that $U_{i}$ and $E_{i}$ are $H_{i}$-vector spaces the space of $H_{i}$-equivariant smooth maps $\mathcal{C}^{\infty}_{H_{i}}(U_{i},E_{i})$ is a $\mathcal{C}^{\infty}_{H_{i}}(U_{i})$-module with finite polynomial generators.
Suppose that
$$F_{1},\,F_{2},\cdot\cdot\cdot,F_{r_{i}}$$
are the generators of $\mathcal{C}^{\infty}_{H_{i}}(U_{i},E_{i})$.
Since a $G$-equivariant map on $W_{i}$ is uniquely determined by its restriction to $U_{i}$ the local section $S|_{W_{i}}$ is uniquely determined by a $H_{i}$-equivariant map $\tilde{S}_{i}\in \mathcal{C}^{\infty}_{H_{i}}(U_{i},E_{i})$.
There exists a unique set of $H_{i}$-invariant smooth functions
$$\textbf{h}=(h_{1},\cdot\cdot\cdot,h_{r_{i}})\in \mathcal{C}^{\infty}_{H_{i}}(U_{i})^{r_{i}}$$
such that
\begin{equation}
\tilde{S}_{i}=\sum^{r_{i}}_{j=1}h_{j}F_{j}.
\end{equation}
The graph map of $\tilde{S}_{i}$ is
\begin{equation}\label{12}
\Gamma(\tilde{S}_{i}):U_{i}\longrightarrow U_{i}\times \mathbb{R}^{r_{i}},x\longmapsto(x,\textbf{h}(x))
\end{equation}
and the universal variety is
$$\mathcal{E}_{i}=\{(x,t)\in U_{i}\times \mathbb{R}^{r_{i}}\mid \sum^{r_{i}}_{j=1}t_{j}F_{j}(x)=0\}.$$
From Definition \ref{def2.2.2}, $\tilde{S}_{i}$ is in general position if and only if (\ref{12}) is transverse to $\mathcal{E}_{i}$ (every stratum of $\mathcal{E}_{i}$) in $U_{i}\times \mathbb{R}^{r_{i}}$.
Given any $\textbf{c}=(c_{1},\cdot\cdot\cdot,c_{r_{i}})\in\mathbb{R}^{r_{i}}$ we can make a perturbation of the graph map (\ref{12}) as follows:
\begin{equation}\label{13}
  x\longmapsto(x,h_{1}(x)+c_{1},\cdot\cdot\cdot,h_{r_{i}}(x)+c_{r_{i}}).
\end{equation}
Since the set of points $\textbf{c}\in\mathbb{R}^{r_{i}}$ such that the map (\ref{13}) is transversal to $\mathcal{E}_{i}$ is dense in $\mathbb{R}^{r_{i}}$ we can choose $H_{i}$-invariant functions
$$\textbf{l}_{i}=(l_{1},\cdot\cdot\cdot,l_{r_{i}})\in \mathcal{C}^{\infty}_{H_{i}}(U_{i})^{r_{i}}$$
such that the map
\begin{equation}
  x\longmapsto(x,h_{1}(x)+l_{1}(x),\cdot\cdot\cdot,h_{r_{i}}(x)+l_{r_{i}}(x))
\end{equation}
is transverse to $\mathcal{E}_{i}$ in $U_{i}\times \mathbb{R}^{r_{i}}$. Let
\begin{equation}
\sigma_{i}=\sum^{r_{i}}_{j=1}l_{j}F_{j},
\end{equation}
then $\tilde{S}_{i}+\sigma_{i}$ is in general position over $U_{i}$.
Furthermore, $\sigma_{i}$ determines a unique $G$-equivariant local section $P_{i}:W_{i}\longrightarrow E$ such that $S\mid_{W_{i}}+P_{i}$ is in general position.
Let $W_{q+1}=B-S^{-1}(0)$, then $W_{q+1}$ is a $G$-invariant open subset since $S^{-1}(0)$ is closed and $G$-invariant.
Notice that
$$W_{0},W_{1},...,W_{q+1}$$
form a $G$-invariant open covering of $B$, there exists a $G$-invariant partition of unity on $B$, i.e. there are $G$-invariant smooth functions
$$\chi_{j}:B\longrightarrow[0,1],0\leq j\leq q+1$$
such that
$$\textmd{supp}(\chi_{j})\subset W_{j},\quad\sum^{q+1}_{j=1}\chi_{j}(x)=1,\,\forall x\in B.$$
Let $P=\sum^{q}_{j=0}\chi_{j}P_{j}$, then $P$ is supported in $\bigcup_{i=0}^{q}W_{i}$, which is a $G$-invariant open neighborhood of $S^{-1}(0)$.
According to the openness and the density of the set of the smooth equivariant sections which are in general position, via the choice of
$$(\textbf{l}_{0},\cdot\cdot\cdot,\textbf{l}_{q})\in\prod^{q}_{i=0}\mathcal{C}^{\infty}_{H_{i}}(U_{i})^{r_{i}}$$
we can make $S+P$ be in general position with respect to the zero section over $B$.

We are now in a position to verify that the zero locus $(S+P)^{-1}(0)\subset B$ represents a homology class in $H_{n-k}(B;\mathbb{Z})$.
Note that $E$ is a $G$-vector bundle over $B$, therefore $B$ can be considered as an embedded $G$-invariant submanifold of $E$.
For the simplicity, let $\hat{S}=S+P$.
Note that $\hat{S}:B\longrightarrow E$ is a $G$-equivariant smooth section which is in general position with respect to the zero section.
Hence, from the result of Bierstone (cf. \cite[Proposition 6.5]{EB77}), the zero locus $X=\hat{S}^{-1}(0)\subset B$ is a compact Whitney object with $G$-invariant submanifolds as its strata. From Proposition \ref{3.0.19}, for each closed subgroup $H$ of $G$ the section $\hat{S}_{H}:B_{(H)}\longrightarrow \mathcal{E}_{H}$ is transverse to the zero section of $\mathcal{E}_{H}$.
Thus we get $\hat{S}^{-1}_{H}(0)=X\cap B_{(H)}$ is a $G$-invariant submanifold with dimension $r_{H}:=\textmd{dim}\,B_{(H)}-\textmd{rank}\,\mathcal{E}_{H}$.
Let $X_{H}=\hat{S}^{-1}_{H}(0)$, then
$$X=\bigsqcup_{H<G}X_{H}.$$
In particular, if $H=e$ is trivial subgroup of $G$, then $B_{(e)}=B_{e}$ is an open subset of $B$ so that $B_{e}$ is oriented and $\textmd{dim}\,B_{(e)}=\textmd{dim}\,B$.
Furthermore, the orientations on $B_{(e)}$ and $E$ determines an orientation on $X_{e}$, i.e. $X_{e}$ is an oriented submanifold with dimension $r_{e}=n-k$. Note that the coindex of $(B,E)$ satisfies $\textmd{coind}(B,E)>1$, we obtain $r_{H}\leq n-k-2$ when $H\neq e$.
At last we get that $X\subset B$ is an oriented compact $G$-invariant Whitney object with dimension $n-k$; especially, there is no codimension one stratum thus the cycle condition is automatic.
Therefore $X$ yields a $G$-invariant $(n-k)$-geometric cycle $\xi_{X}$; moreover, through the representation map (\ref{2}) we get a homology class
$[\xi_{X}]\in H_{n-k}(B;\mathbb{Z})$.

Finally, we have to show that homology class $[\xi_{X}]$ is independent of the choice of $P$.
To prove such independence, we only need to verify that different equivariant perturbations of section $S$ yield the $G$-invariant $(n-k)$-geometric cycles which are cobordant.
In this step we need the following lemma which is a relative version of Theorem \ref{thm2.2.9} and we give its proof at the end of this section for the completeness.
\begin{lem}\label{4.20}
Let $\pi:E\longrightarrow B$ be a $G$-vector bundle, $S:B\longrightarrow E$ be a $G$-equivariant smooth section and $K\subset B$ be a $G$-invariant closed compact subset.
If $S$ is in general position with respect to the zero section over $K$, then there exists a $G$-equivariant smooth section $\tilde{S}$ such that $\tilde{S}$ is in general position with respect to the zero section over $B$, and the restriction of $\tilde{S}$ on $K$ is equivalent to $S$, i.e. $\tilde{S}\mid_{K}=S\mid_{K}$.
\end{lem}

Suppose that $S_{0}$ and $S_{1}$ are two $G$-equivariant smooth sections which are in general position.
Let $X_{0}=S^{-1}_{0}(0)$ and $X_{1}=S^{-1}_{1}(0)$.
Then we get two $G$-invariant $(n-k)$-geometric cycles, denoted by $\xi_{0}$ and $\xi_{1}$ such that $|\xi_{0}|=X_{0}$ and $|\xi_{1}|=X_{1}$.
Let $G$ acts on $\mathbb{R}$ trivially, then $E\times \mathbb{R}$ and $B\times \mathbb{R}$ are two $G$-spaces and we can construct a new $G$-vector bundle of rank $k$ as follows
\[\begin{array}{ccc}
\mathbb{R}^{k}&
\stackrel{}{\longrightarrow} &
E\times \mathbb{R} \\
&&
\Big\downarrow\vcenter{%
\rlap{$\scriptstyle{\pi}\,$}}\\
& &
B\times \mathbb{R}
\end{array}\]
Define a section of the above $G$-vector bundle
\begin{equation}
S:B\times \mathbb{R}\longrightarrow E\times \mathbb{R},\,(x,t)\longmapsto (1-t)S_{0}(x)+tS_{1}(x).
\end{equation}
Clearly $S$ is a $G$-equivariant smooth section and $S(x,0)=S_{0}(x)$, $S(x,1)=S_{1}(x)$.
Note that $S_{0}$ and $S_{1}$ are in general position over $B$, hence $S$ is in general position over the compact and closed subset $K=K_{0}\cup K_{1}$ of $B\times \mathbb{R}$, where $K_{0}=X_{0}\times \{0\}$ and $K_{1}=X_{1}\times \{1\}$.
By the above lemma we can construct a $G$-equivariant smooth section of $E\times \mathbb{R}$, denoted by $\tilde{S}$, which is in general position and $\tilde{S}\mid_{K}=S\mid_{K}$.
Let $X=\tilde{S}^{-1}(0)\cap(B\times [0,1])$, then $X$ is an oriented $(n-k+1)-$dimensional compact Whitney object; moreover, $X$ yields a $(n-k+1)$-geometric chain $\eta$ such that $|\eta|=X\subset B\times [0,1]$.
From the equivariant isotopy theorem (cf. \cite[Theorem 1.5]{EB77}) there exists a $\delta>0$ and an equivariant homeomorphism
\begin{equation}
\Upsilon_{0}:B\times(-\delta,\delta)\longrightarrow B\times(-\delta,\delta)
\end{equation}
covering the identity map, such that the restriction $\Upsilon_{0}|_{B\times\{0\}}$ is the identity map and
\begin{eqnarray*}
\Upsilon_{0}((\tilde{S}|_{B\times(-\delta,\delta)})^{-1}(0))&=&S^{-1}_{0}(0)\times(-\delta,\delta)\\
&=&|\xi_{0}|\times(-\delta,\delta).
\end{eqnarray*}
Similarly, there exists a $\epsilon>0$ and an equivariant homeomorphism
\begin{equation}
\Upsilon_{1}:B\times(1-\epsilon,1+\epsilon)\longrightarrow B\times(1-\epsilon,1+\epsilon)
\end{equation}
such that $\Upsilon_{1}|_{B\times\{1\}}$ is the identity map and
\begin{eqnarray*}
\Upsilon_{1}((\tilde{S}|_{B\times(1-\epsilon,1+\epsilon)})^{-1}(0))&=&S^{-1}_{1}(0)\times(1-\epsilon,1+\epsilon)\\
&=&|\xi_{1}|\times(1-\epsilon,1+\epsilon).
\end{eqnarray*}
Let $\theta=\min\{\delta,\epsilon\}$, then we have
\begin{enumerate}
  \item $|\eta|\cap B\times[0,\theta)=|\xi_{0}|\times[0,\theta)$;
  \item $|\eta|\cap B\times(1-\theta,1]=|\xi_{1}|\times(1-\theta,1]$;
  \item $\partial\eta=\xi_{1}\times\{1\}-\xi_{0}\times\{0\}$.
\end{enumerate}
Thus $\xi_{0}$ and $\xi_{1}$ are cobordant and they represent the same homology class in $H_{n-k}(B;\mathbb{Z})$.
This completes the proof.
\end{proof}

\begin{proof}[Proof of Lemma \ref{4.20}]
Since $B$ is a proper $G$-manifold, for every point $x\in B$ there exists a $G$-invariant open neighborhood of $x$, denoted by $U_{x}\subset B$.
Clearly $K$ has a $G$-invariant open covering
$$K\subset\bigcup_{x\in K}U_{x}.$$
Note that $K$ is compact, there exist finitely many points $x_{1},...,x_{l}$ of $K$ such that
$$K\subset\bigcup^{l}_{i=1}U_{i}$$
where $U_{i}=U_{x_{i}}$. Let $U_{K}=\cup^{l}_{i=1}U_{i}$, then $U_{K}$ is a $G$-invariant open neighborhood of $K$.
As $K$ is  $G$-invariant and closed, $U_{0}:=B-K$ is a $G-$invariant open subset of $B$.
It follows that
$$U_{0},U_{1},...,U_{l}$$
forms a finite open covering of $B$ with $G$-invariant open subsets.
Using the $G$-invariant partition of unity on proper $G$-manifold we have the $G$-invariant smooth functions
$$\chi_{i}:B\longrightarrow [0,1],0\leq i\leq l$$
such that $supp(\chi_{i})\subset U_{i}$, and $\sum^{l}_{i=0}\chi_{i}(x)=1$ for any $x$ of $B$.
Furthermore, we get
\begin{eqnarray*}
\chi_{0}(x)=\left\{\begin{array}{l}
0;x\in K\\
1;x\in B-U_{K}
\end{array}
\right.
\end{eqnarray*}
\begin{eqnarray*}
\sum^{l}_{i=1}\chi_{i}(x)=\left\{\begin{array}{l}
1;x\in K\\
0;x\in B-U_{K}
\end{array}
\right.
\end{eqnarray*}
Assume that $\rho_{0}=\chi_{0}$ and $\rho_{1}=\sum^{l}_{i=1}\chi_{i}$.
For any $G$-equivariant smooth section $S^{\prime}$ which is in general position with respect to the zero section over $B$, let $\tilde{S}=\rho_{0}S^{\prime}+\rho_{1}S$, then $\tilde{S}$ is a $G$-equivariant smooth section.
In particular, $\tilde{S}$ is in general position over the invariant closed subset $K\cup(B-U_{K})$ and $\tilde{S}\mid_{K}=S\mid_{K}$.
In fact, it is in general position over an invariant neighborhood of $K\cup(B-U_{K})$ since a smooth equivariant map in general position at a point implies that it is in general position over an invariant neighborhood of this point (cf. \cite[Lemma 6.2 and Proposition 6.3]{EB77}).
According to the density we can choose $S^{\prime}$ such that $\tilde{S}$ is in general position over $U_{K}-K$.
\end{proof}

\begin{ex} Given three coprime integers $p_{0}$, $p_{1}$ and $p_{2}$, consider
$$B=S^{5}=\biggl\{(z_{0},z_{1},z_{2})\in \mathbb{C}^{3}\mid\sum^{2}_{i=0}\mid z_{i}\mid^{2}=1\biggr\}.$$
Let $G=S^{1}$ acts on $B$ by
$$\lambda(z_{0},z_{1},z_{2})=(\lambda^{p_{0}}z_{0},\lambda^{p_{1}}z_{1},\lambda^{p_{2}}z_{2}),\quad \forall \lambda\in S^{1}.$$
Then the singular strata of the orbit type stratification of $B$ are
$$B_{0}:=\biggl\{b\in B\mid G_{b}=\mathbb{Z}_{p_{0}}\biggr\}\cong S^{1}/\mathbb{Z}_{p_{0}};$$
$$B_{1}:=\biggl\{b\in B\mid G_{b}=\mathbb{Z}_{p_{1}}\biggr\}\cong S^{1}/\mathbb{Z}_{p_{1}};$$
$$B_{2}:=\biggl\{b\in B\mid G_{b}=\mathbb{Z}_{p_{2}}\biggr\}\cong S^{1}/\mathbb{Z}_{p_{2}}.$$
Let $\pi:E\longrightarrow B$ be a $S^{1}$-equivariant plane bundle.
Note that the codimensions of $B_{0}$, $B_{1}$ and $B_{2}$ are $4$, and moreover, the rank of obstruction bundle over each singular stratum is smaller than or equivalent to $2$, and by the above theorem we may obtain the invariant Euler cycle via equivariant perturbation.
\end{ex}
\begin{ex}[An application in symplectic geometry]
Theorem \ref{Main Thm1} can be applied to the study of symplectic geometry.
Assume that $(M,\omega,J)$ is a compact spherically positive symplectic manifold and $L\subset M$ is a relatively spin Lagrangian submanifold, and moreover, let $\beta\in H_{2}(M,L;\mathbb{Z})$.
In the recent research paper \cite{FOOO13} Fukaya-Oh-Ohta-Ono studied the moduli space of stable $(k+1)$-marked pseudo-holomorphic discs with respect to $L$ and $\beta$.
They proved that there exists an oriented Kuranishi structure on the moduli space $\mathcal{M}^{main}_{k+1}(\beta;P_{1},\cdot\cdot\cdot,P_{k})$
\footnote{For the details of the definition of Lagrangian Floer moduli space and its Kuranishi structure refer to the paper \cite[Sections 2 and 11]{FOOO13}.}, and furthermore, they developed the Lagrangian Floer theory over $\mathbb{Z}$ coefficients (see \cite[Theorem 1.1]{FOOO13}).

To develop the Lagrangian Floer theory over $\mathbb{Z}$, the main technique is to construct a single-valued perturbation of the moduli space, which can give rise to a virtual moduli cycle over $\mathbb{Z}$.
Using the notion of the sheaf of groups and the notion of normal bundles in the sense of stacks the authors constructed a suitable single-valued perturbation of the Kuranishi structure (see \cite[Theorem 3.1]{FOOO13}).
Applying this single-valued perturbation to the concrete Lagrangian Floer moduli space of the spherically positive symplectic manifold, they constructed the Lagrangian Floer theory over integers.

In fact, for any point $p\in\mathcal{M}^{main}_{k+1}(\beta;P_{1},\cdot\cdot\cdot,P_{k})$ the Kuranishi chart associated to $p$ is an oriented $\Gamma_{p}$-equivariant moduli problem $(V_{p},E_{p},S_{p})$ with coindex $>$1, where $\Gamma_{p}$ is a finite group (see \cite[Proposition 12.1]{FOOO13}). More precisely, $E_{p}$ is a $\Gamma_{p}$-equivariant vector bundle over $V_{p}$ and $S_{p}:V_{p}\rightarrow E_{p}$ is a smooth $\Gamma_{p}$-equivariant section.
Note that $\textmd{coind}\,(V_{p},E_{p})>1$, therefore, by Theorem \ref{Main Thm1} we may construct a $\Gamma_{p}$-equivariant perturbation of $S_{p}$ such that the perturbed section, denoted by $S^{\prime}_{p}$, is in general position with respect to the zero section of $E_{p}$ over $V_{p}$.
Furthermore, we can construct a global perturbation of the moduli space by gluing together those local equivariant perturbations over Kuranishi charts in a compatible way.
In particular, this global perturbation may yield a geometric cycle with dimension equal to the virtual dimension of the moduli space.
\end{ex}

\section{Transversal intersection of $S^{1}$-moduli problems}
In this section we study the intersection problem of $S^{1}$-moduli problems.
\begin{defn}[Goresky \cite{RMG81}] \label{def5.0.23}
Let $X$ be a fixed Whitney object.
Assume that $V$ and $W$ are two substratified objects in $X$.
We say $V$ is \emph{transverse} to $W$ provided that for every stratum $R\subset V$ and every stratum $S\subset W$ satisfy:
(1) $R\cap S=\emptyset$ or;
(2) $R$ is transverse to $S$ in the stratum $X_{i}\subset X$ which contains $R$ and $S$.
\end{defn}

For a compact smooth $n$-manifold $B$ on which $G=S^{1}$ acts, there exists a canonical Whitney stratification on $B$ determined by orbit types.
For the simplicity we assume that the $G$-action is semi-free and the $G$-fixed loci is connected.
With this assumption there exist only two orbit types.
Let
$$B_{0}=\{x\in B| G_{x}=e\},\quad B_{1}=\{x\in B| G_{x}=G\}$$
then $B=B_{0}\sqcup B_{1}$. Let $(B,E_{\alpha},S_{\alpha})$ and $(B,E_{\beta},S_{\beta})$ be two oriented $G$-moduli problems such that $\textmd{rank}\,E_{\alpha}=k$ and $\textmd{rank}\,E_{\beta}=n-k$.
Assume that $S_{\alpha}$ and $S_{\beta}$ are in general position then the associated moduli spaces
$$M_{\alpha}=\{x\in B| S_{\alpha}(x)=0\},\quad M_{\beta}=\{x\in B| S_{\beta}(x)=0\}$$
are Whitney substratified objects with $G$-invariant strata in $B$.
Let
$$M_{\alpha,0}= M_{\alpha}\cap B_{0},\quad M_{\alpha,1}= M_{\alpha}\cap B_{1}$$
and
$$M_{\beta,0}= M_{\beta}\cap B_{0},\quad M_{\beta,1}= M_{\beta}\cap B_{1}$$
then the Whitney stratifications induced by orbit types on $M_{\alpha}$ and $M_{\beta}$ are
$$M_{\alpha}=M_{\alpha,0}\sqcup M_{\alpha,1},\qquad M_{\beta}=M_{\beta,0}\sqcup M_{\beta,1}.$$

Denote the partitions of $(B,E_{\alpha},S_{\alpha})$ and $(B,E_{\beta},S_{\beta})$ by
$$(B,E_{\alpha},S_{\alpha})=(B_{0},\mathcal{E}_{\alpha,0},S_{\alpha,0})\sqcup(B_{1},\mathcal{E}_{\alpha,G},S_{\alpha,G})$$
and
$$(B,E_{\beta},S_{\beta})=(B_{0},\mathcal{E}_{\beta,0},S_{\beta,0})\sqcup(B_{1},\mathcal{E}_{\beta,G},S_{\beta,G}).$$
Suppose that $M_{\alpha}$ is transverse to $M_{\beta}$, by Definition \ref{def5.0.23} we get that $M_{\alpha,0}$ is transverse to $M_{\beta,0}$ in $B_{0}$ and $M_{\alpha,1}$ is transverse to $M_{\beta,1}$ in $B_{1}$.
Since $B_{0}\subset B$ is an open subset, $\textmd{dim}\,B_{0}=n$.
Observe that $M_{\alpha,0}$ is $(n-k)$-dimensional and $M_{\beta,0}$ is $k$-dimensional, if the transversal intersection $M_{\alpha,0}\cap M_{\beta,0}$ is non-trivial, i.e. $M_{\alpha,0}\cap M_{\beta,0}\neq\emptyset$, then $M_{\alpha,0}\cap M_{\beta,0}$ is an invariant submanifold with dimension 0.
In other aspect, for any $z\in M_{\alpha,0}\cap M_{\beta,0}$ the orbit $G(z)$ belongs to $M_{\alpha,0}\cap M_{\beta,0}$ since $M_{\alpha,0}\cap M_{\beta,0}$ is $G$-invariant.
Note that the isotropy subgroup $G_{z}=e$, $G(z)$ is isomorphic to $G=S^{1}$.
So
$$\textmd{dim}\,(M_{\alpha,0}\cap M_{\beta,0})\geq \textmd{dim}\,G(z)=1$$
and this leads to a contradiction with $\textmd{dim}\,(M_{\alpha,0}\cap M_{\beta,0})=0$.
Hence $M_{\alpha,0}\cap M_{\beta,0}=\emptyset$ and we have
$$Z:=M_{\alpha}\cap M_{\beta}=M_{\alpha,1}\cap M_{\beta,1}$$
is a submanifold of $B_{1}$.

From now on, we assume that the $G$-fixed subbundles $\mathcal{E}_{\alpha,G}\longrightarrow B_{1}$ and $\mathcal{E}_{\beta,G}\longrightarrow B_{1}$ are oriented.
Consider the obstruction bundles $\mathfrak{o}_{\alpha}:\mathcal{O}_{\alpha,1}\longrightarrow B_{1}$ and $\mathfrak{o}_{\beta}:\mathcal{O}_{\beta,1}\longrightarrow B_{1}$.
Suppose that
$$\textmd{rank}\,\mathcal{O}_{\alpha,1}=k-n_{\alpha}$$
and
$$\textmd{rank}\,\mathcal{O}_{\beta,1}=n-k-n_{\beta},$$
where $n_{\alpha}$ and $n_{\beta}$ are the ranks of the $G$-fixed subbundles $\mathcal{E}_{\alpha,G}$ and $\mathcal{E}_{\beta,G}$ respectively.
The orientations of $\mathcal{E}_{\alpha,G}$ and $\mathcal{E}_{\beta,G}$ induce the orientations of the obstruction bundles.

Let $\textmd{dim}\,B_{1}=n_{1}$.
Observe that $S_{\alpha}:B\longrightarrow E_{\alpha}$ is in general position, the section
$$S_{\alpha,G}=S_{\alpha}|_{B_{1}}:B_{1}\longrightarrow \mathcal{E}_{\alpha,G}$$
is transverse to the zero section of the $G$-fixed bundle $\mathcal{E}_{\alpha,G}$ and the zero locus
$$M_{\alpha,1}=S^{-1}_{\alpha,G}(0)\subset B_{1}$$
is an oriented submanifold with dimension $n_{1}-n_{\alpha}$.
Similarly,
$$M_{\beta,1}=S^{-1}_{\beta,G}(0)\subset B_{1}$$
is an oriented $(n_{1}-n_{\beta})$-dimensional submanifold.
Note that $M_{\alpha}$ and $M_{\beta}$ are Whitney substratified objects, however, on the level of set the intersection set $Z=M_{\alpha}\cap M_{\beta}$ is a submanifold of fixed loci $B_{1}$ with dimension $n_{1}-n_{\alpha}-n_{\beta}$.

Consider the direct sum of $(B,E_{\alpha},S_{\alpha})$ and $(B,E_{\beta},S_{\beta})$.
Let $E=E_{\alpha}\oplus E_{\beta}$ and $S=S_{\alpha}\oplus S_{\beta}$, then we get a new oriented $G$-moduli problem $(B,E,S)$ with $\textmd{dim}\,B=\textmd{rank}\,E=n$.
The associated partition of $(B,E,S)$ is
$$(B,E,S)=(B_{0},\mathcal{E}_{0},S_{0})\sqcup(B_{1},\mathcal{E}_{G},S_{G})$$
i.e.
\[
\begin{aligned}
&\phantom{\mathcal{M}} \\
&\phantom{\mathcal{M}} \\
 \end{aligned}
 \qquad\qquad
\xymatrix{
 \mathcal{E}_{0}   \ar@(ul,dl)_{\textstyle  G}\ar@{->}[d]     \\
 B_{0} \ar@(ul,dl)_{\textstyle G} \ar@/_1pc/[u]_{S_{0}}
}
\begin{aligned}
&\phantom{\mathcal{M}} \\
&\phantom{\mathcal{M}} \\
 \end{aligned}
 \qquad\qquad
\xymatrix{
 \mathcal{E}_{G}   \ar@(ul,dl)_{\textstyle  G}\ar@{->}[d]     \\
 B_{1} \ar@(ul,dl)_{\textstyle G} \ar@/_1pc/[u]_{S_{G}}
}
\]
where
$$\mathcal{E}_{0}=\mathcal{E}_{\alpha,0}\oplus \mathcal{E}_{\beta,0},\quad S_{0}=S_{\alpha,0}\oplus S_{\beta,0}$$
and
$$\mathcal{E}_{G}=\mathcal{E}_{\alpha,G}\oplus \mathcal{E}_{\beta,G},\quad S_{G}=S_{\alpha,G}\oplus S_{\beta,G}.$$
The obstruction bundle over $B_{1}$ is
$$\mathfrak{o}:\mathcal{O}_{1}\longrightarrow B_{1}$$
where $\mathcal{O}_{1}=\mathcal{O}_{\alpha,1}\oplus \mathcal{O}_{\beta,1}$ and $\mathfrak{o}=\mathfrak{o}_{\alpha}\oplus\mathfrak{o}_{\beta}$.

Note that $S_{\alpha,0}$ and $S_{\beta,0}$ are transverse to the zero sections of $\mathcal{E}_{\alpha,0}$ and $\mathcal{E}_{\beta,0}$ respectively and $M_{\alpha,0}=S^{-1}_{\alpha,0}(0)$ intersects with $M_{\beta,0}=S^{-1}_{\beta,0}(0)$ in $B_{0}$ transversally.
This implies that $S_{0}$ is transverse to the zero section of $\mathcal{E}_{0}$ and we get
$$S^{-1}_{0}(0)=M_{\alpha,0}\cap M_{\beta,0}=\emptyset.$$
Similarly, we obtain that $S_{G}$ is transverse to the zero section of $\mathcal{E}_{G}$ and
$$
S^{-1}_{G}(0)=M_{\alpha,1}\cap M_{\beta,1}=Z.
$$

Define
$$\Psi(E_{\alpha},E_{\beta})=\int_{B}e(E_{\alpha}\oplus E_{\beta}).$$
We call $\Psi(E_{\alpha},E_{\beta})$ the \emph{intersection number} of $G$-moduli problems $(B,E_{\alpha},S_{\alpha})$ and $(B,E_{\beta},S_{\beta})$.
Firstly, we consider the non-degenerated case, i.e. $S_{\alpha}$ and $S_{\beta}$ are transverse to the zero sections and the moduli spaces $M_{\alpha}$ intersect with $M_{\beta}$ in $B$ transversally. In this case $M_{\alpha}\subset B$ is an oriented $G$-invariant submanifold of dimension $n-k$, and $M_{\beta}\subset B$ is a $k$-dimensional invariant submanifold.
Since $M_{\alpha}$ is transverse to $M_{\beta}$ in $B$ the intersection number is
$${\#}(M_{\alpha}\cdot M_{\beta})=\int_{B}\textmd{PD}(M_{\alpha})\wedge \textmd{PD}(M_{\beta})$$
where $\textmd{PD}(\cdot)$ is the Poincar\'{e} dual.
Note that $\textmd{PD}(M_{\alpha})=e(E_{\alpha})$ and $\textmd{PD}(M_{\beta})=e(E_{\beta})$, and hence, we get
\begin{eqnarray*}
\int_{B}\textmd{PD}(M_{\alpha})\wedge \textmd{PD}(M_{\beta})&=&\int_{B}e(E_{\alpha})\wedge e(E_{\beta})\\
&=&\int_{B}e(E_{\alpha}\oplus E_{\beta}).
\end{eqnarray*}
Therefore, in the non-degenerated case the intersection number of $G$-moduli problems is equivalent to the intersection number of the associated moduli spaces, i.e.
$$\Psi(E_{\alpha}, E_{\beta})= \,{\#}(M_{\alpha}\cdot M_{\beta}).$$

In general, the existence of the obstruction system of $G$-moduli problem implies that the equivariant smooth section which is transverse to the zero section do not always exist.
However, the equivariant sections which are in general position are generic, and especially we have:
\begin{thm}\label{Main Thm2}
Assume that $S_{\alpha}:B\longrightarrow E_{\alpha}$ and $S_{\beta}:B\longrightarrow E_{\beta}$ are in general position with respect to the zero sections
respectively.
If the $G$-moduli space $M_{\alpha}$ is transverse to the $G$-moduli space $M_{\beta}$ in the sense of Definition \ref{def5.0.23} then
$$\Psi(E_{\alpha}, E_{\beta})=\int_{Z}i^{*}\biggl(\frac{e_{G}(\mathcal{O}_{1})}{e_{G}(\mathcal{N}_{B_{1}/B})}\biggr)$$
where $Z=M_{\alpha}\cap M_{\beta}$, $\mathcal{N}_{B_{1}/B}$ is the normal bundle of $B_{1}$ in $B$ and $i^{*}$ is the map induced by the inclusion
$i:Z\hookrightarrow B_{1}$.
\end{thm}

\begin{proof}
Let $\Theta_{G}\in \Omega^{n}_{G,vc}(E)$ be the equivariant Thom form of $E$.
By the definition of equivariant Thom form (cf. \cite[Theorem 6.4]{MQ86}), the leading component of $\Theta_{G}$, denoted by
$$\Theta=(\Theta_{G})_{[n]}\in \Omega^{n}_{vc}(E),$$
is a non-equivariant Thom form of $E$. Denote by
$i_{0}:B\longrightarrow E$
the embedding of $B$ in $E$ as the zero section, then the equivariant Euler class of $E$ is $e_{G}(E)=i^{*}_{0}(\Theta_{G})$ and the ordinary one is $e(E)=i^{*}_{0}(\Theta)$.
Observe that $\textmd{dim}\,B=n$, according to the definition of equivariant integral we have
\begin{eqnarray*}
\int_{B}e_{G}(E)&=&\int_{B}i^{*}_{0}(\Theta_{G})\\
&=&\int_{B}i^{*}_{0}((\Theta_{G})_{[n]})\\
&=&\int_{B}e(E),
\end{eqnarray*}
and therefore we get
\begin{equation}
  \Psi(E_{\alpha},E_{\beta})=\int_{B}e_{G}(E).
\end{equation}
Let $\mathcal{N}_{B_{1}/B}$ be the normal bundle of $B_{1}$ in $B$.
The $G$-action on $\mathcal{N}_{B_{1}/B}$ only fixes the zero section $B_{1}$.
This implies that the normal bundle $\mathcal{N}_{B_{1}/B}$ has even rank and is orientable.
In particular, with a fixed orientation the equivariant Euler class of the normal bundle $e_{G}(\mathcal{N}_{B_{1}/B})$ is invertible.
Using the Atiyah-Bott-Berline-Vergne localization formula (cf. \cite[Theorem C.53]{VVY97}) we have
\begin{equation}\label{14}
  \int_{B}e_{G}(E)=\int_{B_{1}}\frac{j^{*}e_{G}(E)}{e_{G}(\mathcal{N}_{B_{1}/B})}
\end{equation}
where $j^{*}$ is the map induced by the inclusion $j:B_{1}\hookrightarrow B$.
Note that $j^{*}e_{G}(E)=e_{G}(j^{*}E)$ and the pullback bundle $j^{*}E$ is equivalent to $E|_{B_{1}}$.
In other aspect, $E|_{B_{1}}$ can split into the direct sum of $G$-fixed subbundle and obstruction bundle which are all $G$-equivariant, i.e. $E|_{B_{1}}=\mathcal{E}_{G}\oplus \mathcal{O}_{1}$.
Using the equivariant Chern-Weil theory (cf. \cite[Chapter 8]{VGSS99}), and just following the proof of Whitney product formula for Euler class we get
\begin{equation}\label{15}
  j^{*}e_{G}(E)=e_{G}(\mathcal{E}_{G})\wedge e_{G}(\mathcal{O}_{1}).
\end{equation}
Consider the equivariant Euler class $e_{G}(\mathcal{E}_{G})$.
Note that the equivariant section $S_{G}:B_{1}\longrightarrow \mathcal{E}_{G}$ is transverse to the zero section, the zero locus $Z=S^{-1}_{G}(0)\subset B_{1}$ is an invariant submanifold.
Moreover, the normal bundle of $Z$ in $B_{1}$, denoted by $\mathcal{N}_{Z/B_{1}}$, is isomorphic to $\mathcal{E}_{G}\mid_{Z}$.
Assume that $\iota:B_{1}\hookrightarrow \mathcal{E}_{G}$ is the embedding of $B_{1}$ into $\mathcal{E}_{G}$ as the zero section.
Without loss of generality, we may choose an equivariant Thom form of $\mathcal{E}_{G}$, denoted by
$\Phi_{G}\in \Omega^{n_{\alpha}+n_{\beta}}_{G,vc}(\mathcal{E}_{G})$, such that the support of the pullback by $S_{G}$
$$S^{*}_{G}(\Phi_{G})\in \Omega^{n_{\alpha}+n_{\beta}}_{G}(B_{1}),$$
is contained in an invariant tubular neighborhood of $Z$ in $B_{1}$. Let
$$\Phi=(\Phi_{G})_{[n_{\alpha}+n_{\beta}]}\in \Omega^{n_{\alpha}+n_{\beta}}_{vc}(\mathcal{E}_{G})$$
be the leading component of $\Phi_{G}$, then $\Phi$ is a non-equivariant Thom form of $\mathcal{E}_{G}$.
Given any $z\in Z$, let $N_{z}$ be the fibre of $\mathcal{N}_{Z/B_{1}}$ at $z$ and $E_{G,z}$ be the fibre of $\mathcal{E}_{G}$ at $z$.
Because the image of a fibre of $\mathcal{N}_{Z/B_{1}}$ under $S_{G}$ is homotopic to a fibre of $\mathcal{E}_{G}$ we have
\begin{eqnarray*}
\int_{N_{z}}S^{*}_{G}(\Phi_{G})&=&\int_{N_{z}}S^{*}_{G}(\Phi)\quad(\textmd{dim}\,N_{z}=n_{\alpha}+n_{\beta})\\
&=&\int_{E_{G,z}}\Phi\quad\quad\,\,(\Phi\,\, is\,\, \textmd{Thom}\,\, \textmd{form})\\
&=&1.
\end{eqnarray*}
Note that $e_{G}(\mathcal{E}_{G})=\iota^{*}(\Phi_{G})$, the next thing to do is to verify that $\iota^{*}(\Phi_{G})$ is an equivariant Thom form of $\mathcal{N}_{Z/B_{1}}$.
The proof is straightforward since we have
\begin{eqnarray*}
\int_{N_{z}}\iota^{*}(\Phi_{G})&=&\int_{N_{z}}\iota^{*}(\Phi)\quad\,\,(\textmd{dim}\,N_{z}=n_{\alpha}+n_{\beta})\\
&=&\int_{N_{z}}S^{*}_{G}(\Phi)\quad(\iota^{*}(\Phi)-S^{*}_{G}(\Phi)\,\, is\,\, d-\textmd{exact})\\
&=&1.
\end{eqnarray*}
Denote by $\textmd{PD}_{G}(Z)$ the equivariant Poincar\'{e} dual of $Z$ in $B_{1}$, which is defined as an equivariant Thom form of the normal bundle $\mathcal{N}_{Z/B_{1}}$.
This follows that
\begin{equation}\label{16}
  e_{G}(\mathcal{E}_{G})=\textmd{PD}_{G}(Z).
\end{equation}
Let $i:Z\hookrightarrow B_{1}$ be the inclusion, combining (\ref{14}), (\ref{15}) and (\ref{16}) we get
\begin{eqnarray*}
\int_{B}e_{G}(E)&=&\int_{B_{1}}\frac{\textmd{PD}_{G}(Z)\wedge e_{G}(\mathcal{O}_{1})}{e_{G}(\mathcal{N}_{B_{1}/B})}\\
&=&\int_{Z}i^{*}\biggl(\frac{e_{G}(\mathcal{O}_{1})}{e_{G}(\mathcal{N}_{B_{1}/B})}\biggr).
\end{eqnarray*}
\end{proof}



\end{document}